\newif\ifArXiV

\ArXiVtrue	%ArXiV style

\ifArXiV %---------------------------------------------------------
\documentclass[a4paper]{article}
\usepackage[margin=1.3in]{geometry}
\usepackage[affil-it]{authblk}
\usepackage{amsthm}
\newtheorem{theorem}{Theorem}

\newenvironment{keywords}{\small \textbf{Keywords:}}{}
\newenvironment{acknowledgements}{\small \textbf{Acknowledgements:}}{}
\else %---------------------------------------------------------
\RequirePackage{fix-cm}
\documentclass[smallextended]{svjour3}       % onecolumn (second format)
\smartqed  % flush right qed marks, e.g. at end of proof
\fi%---------------------------------------------------------

\usepackage{graphicx}
\usepackage{amssymb, amsmath, amsfonts}
\usepackage{graphicx}
\usepackage{xcolor}
\usepackage{esvect}
\usepackage{hyperref}
\usepackage{todonotes}
\usepackage{xspace}
\usepackage{subcaption}
\usepackage{multirow}
\usepackage[lined,algoruled,commentsnumbered,noend]{algorithm2e}
\usepackage{marvosym}%for corresponding author letter
\usepackage[normalem]{ulem} %for \sout

\newtheorem{Assumption}{Assumption}
\newtheorem{Observation}{Observation}

\newcommand{\Z}{\mathbb{Z}}
\newcommand{\R}{\mathbb{R}}

\DeclareMathOperator{\st}{s.t.}
\DeclareMathOperator{\conv}{conv}
\newcommand{\red}[1]{{\color{red}#1}}

\newcommand{\BPs}{BPs\xspace}
\newcommand{\MIBLOP}{MIBLP\xspace}
\newcommand{\MIBLOPs}{MIBLPs\xspace}

\newcommand{\VFR}{VFR\xspace}
\newcommand{\HPR}{HPR\xspace}
\newcommand{\BB}{B\&B\xspace}
\newcommand{\BC}{B\&C\xspace}

\newcommand{\IBNP}{IBNP\xspace}
\newcommand{\IBNPs}{IBNPs\xspace}
\newcommand{\SOCP}{SOCP\xspace}
\def\yy{\hat{y}}
\def\xh{\hat{x}}
\newcommand{\exSymb}[1]{\hat{#1}}
\newcommand{\linq}{g}
\newcommand{\DCs}{DCs\xspace}
\newcommand{\DC}{DC\xspace}
\newcommand{\MIX}{\texttt{MIX++}\xspace}

\newcommand{\textoverline}[1]{$\overline{\mbox{#1}}$}
\newcommand{\ccrHPR}{\textoverline{\HPR}\xspace}

\begin{document}

\title{On SOCP-based disjunctive cuts for solving a class of integer bilevel nonlinear programs%
\thanks{An extended abstract of this work
containing the theoretical foundations for a single linking constraint  appeared as~\cite{Gaar-et-al:2022}. The present article additionally handles multiple linking constraints, and presents more theoretical details, new enhancements and many new computational results.}}

\ifArXiV %---------------------------------------------------------
\author[1]{Elisabeth Gaar$^{(\text{\Letter})}$\thanks{elisabeth.gaar@jku.at}}
\author[2]{Jon Lee\thanks{jonxlee@umich.edu}}
\author[3]{Ivana Ljubi\'c\thanks{ljubic@essec.edu}}
\author[1,4]{\\Markus Sinnl\thanks{markus.sinnl@jku.at}}
\author[1]{K\"ubra Tan{\i}nm{\i}\c{s}\thanks{kuebra.taninmis\_ersues@jku.at}}
\affil[1]{Institute of Production and Logistics Management, 
Johannes Kepler University Linz, Linz, Austria}
\affil[2]{University of Michigan, Ann Arbor, Michigan, USA}
\affil[3]{ESSEC Business School of Paris, France}
\affil[4]{JKU Business School, Johannes Kepler University 
Linz, Linz, Austria}			
\date{}

\else %---------------------------------------------------------
\author{Elisabeth Gaar$^{(\text{\Letter})}$ 
%\orcidID{0000-0002-1643-6066} 
\and
Jon Lee  
%\orcidID{0000-0002-8190-1091} 
\\
\and 
Ivana Ljubi\'c
%\orcidID{0000-0002-4834-6284} 
\and
Markus Sinnl 
%\orcidID{0000-0003-1439-8702} 
\\
\and  
K\"ubra Tan{\i}nm{\i}\c{s}   %\orcidID{0000-0003-1081-4182} 
}
\institute{
E. Gaar, M. Sinnl, K. Tan{\i}nm{\i}\c{s}   \at
	Institute of Production and Logistics Management, Johannes Kepler University Linz, Austria,  and
	JKU Business School, Johannes Kepler University Linz, Austria\\
	\email{\{elisabeth.gaar, markus.sinnl, kuebra.taninmis\_ersues\}@jku.at}\\
	J. Lee \at University of Michigan, Ann Arbor, Michigan, USA,
	\email{jonxlee@umich.edu}\\
	I. Ljubi\'c \at
	ESSEC Business School of Paris, France, \email{ljubic@essec.edu}
}

\titlerunning{On SOCP-based DCs for solving a class of \IBNPs}        % if too long for running head

\authorrunning{E. Gaar et al.}

% \author{First Author         \and
%         Second Author %etc.
% }

% %\authorrunning{Short form of author list} % if too long for running head

% \institute{F. Author \at
%               first address \\
%               Tel.: +123-45-678910\\
%               Fax: +123-45-678910\\
%               \email{fauthor@example.com}           %  \\
% %             \emph{Present address:} of F. Author  %  if needed
%               second address
% }

\date{Received: date / Accepted: date}

\fi %---------------------------------------------------------

%MSC classification
%90C11 - Mixed integer programming
%90C57 - Polyhedral combinatorics, branch-and-bound, branch-and-cut
%90C30 - Nonlinear programming
%65K05 - Numerical mathematical programming methods
%\subclass{90C11 \and 90C57 \and 90C30 \and 65K05}

\maketitle
%\setcounter{tocdepth}{3}
%\tableofcontents

\begin{abstract}

We study a class of integer bilevel programs with second-order cone constraints at the upper-level and a convex-quadratic objective function and linear constraints at the lower-level. We develop disjunctive cuts (\DCs) to separate bilevel-infeasible solutions using a second-order-cone-based cut-generating procedure. We propose  \DC separation strategies and consider several approaches for  removing redundant disjunctions and normalization. Using these \DCs, we propose a branch-and-cut algorithm for the problem class we study, and a cutting-plane method for the problem variant with only binary variables.

We present an extensive computational study on a diverse set of instances, including instances with binary and with integer variables, and instances with a single and with multiple linking constraints. Our computational study demonstrates that
the proposed enhancements of our solution approaches are effective for improving the performance. Moreover, both of our approaches outperform a state-of-the-art generic solver for mixed-integer bilevel linear programs that is able to solve a linearized version of our binary instances.

\ifArXiV %---------------------------------------------------------
\else %--------------------------------------------------------------
\fi %--------------------------------------------------------------
\end{abstract}

\section{Introduction}
%\todo[inline]{Write introduction here}

Bilevel programs (\BPs) are challenging hierarchical optimization problems in which the feasible solutions of the
upper-level problem depend on the optimal solution of the lower-level problem. \BPs allow us to model %game-theoretic considerations, notably, 
two-stage two-player Stackelberg games in which two %or more
rational players
(often called \emph{leader} and \emph{follower}) compete in a sequential fashion. 
%The leader, who acts first, tries to anticipate the optimal response of the follower by solving the lower-level optimization problem. 
\BPs have applications in many different domains such as machine learning
%\cite{agor2019feature,franceschi2018bilevel,louati2021deep}
\cite{agor2019feature}, 
logistics 
%\cite{chalmardi2019bi,fontaine2020population,zheng2018exact},
\cite{fontaine2020population},
revenue 
management 
%\cite{cote2003bilevel,labbe2016bilevel},
\cite{labbe2016bilevel}, 
the energy sector 
%\cite{grimm2021optimal,martelli2020optimization,plein2021bilevel} 
\cite{grimm2021optimal,plein2021bilevel}
and portfolio optimization~\cite{gonzalez2021global}. For more details about 
\BPs see, e.g., 
the book by Dempe and Zemkoho~\cite{dempe2020bilevel} and the  recent 
surveys~\cite{Beck-et-al:2022,kleinert2021survey,smith2020survey}.

In this work, %we propose disjunctive cuts (see, e.g., \cite{balas_book}) for solving 
we consider the following %mixed-integer
integer bilevel nonlinear programs with convex leader and follower objective functions (\IBNPs)
\begin{subequations} \label{eq:IBNP_UL}
\label{bilevel}
\begin{align}
&\min %_{\red{x,y}} 
~c'x + d'y \label{eq:objective}\\ %\ell(x,y) \\
%&\mbox{subject to:} \nonumber \\
\st~ &Mx+Ny \geq h \label{eq:coupling}\\
&\tilde{M}x+\tilde{N}y -  \tilde{h} \in \mathcal{K} \label{eq:conic}\\
%&y\in \arg \min\left\{ q(y)  : Ax+By \geq f,~ y \in \mathcal{Y}, ~y \in \mathbb Z^{n_2}  \right\}\\
& y \in \Omega(x) \\
&x \in \mathbb Z^{n_1}, \label{eq:x}
\end{align}
\end{subequations}
where $\Omega(x)$ is the set of optimal solutions of the $x$-parametrized so-called \emph{follower} (or \emph{lower-level}) \emph{problem} 
\begin{align} \label{eq:IBNP_LL}
    \min\left\{ q(y)  : Ax+By \geq f,~ y \in \mathcal{Y}, ~y \in \mathbb Z^{n_2}  \right\}.
\end{align} 
Problem~\eqref{eq:IBNP_UL} is the so-called \emph{leader} (or \emph{upper-level}) \emph{problem}.
The decision variables $x$ and $y$ are of dimension $n_1$ and $n_2$, respectively, and $n:=n_1+n_2$.
%x_i  \in \mathbb Z \quad \forall i \in I
% y_j \in \mathbb Z \quad \forall j \in J
%and $I\subseteq\{1,\ldots,n_1\}$, $J\subseteq\{1,\ldots,n_2\}$.
%
Moreover, we have
$c \in \R^{n_1}$, 
$d \in \R^{n_2}$, 
$M \in \R^{m_1 \times n_1}$, 
$N \in \R^{m_1 \times n_2}$, 
$h \in \R^{           m_1}$, 
$\tilde{M} \in \R^{\tilde{m}_1 \times n_1}$, 
$\tilde{N} \in \R^{\tilde{m}_1 \times n_2}$, 
$\tilde{h} \in \R^{\tilde{m}_1}$, 
$A \in \Z^{m_2 \times n_1}$, 
$B \in \Z^{m_2 \times n_2}$, and 
$f \in \Z^{m_2 }$. 
We denote by $A^i$, $B^i$ and $f_i$ the $i$-th row of $A$, $B$ and $i$-th entry of $f$, respectively.
We assume that each $A^i$ and $B^i$ has at least one non-zero entry.  
The constraints $Ax+By \geq f$ are referred to as \emph{linking constraints}. The constraints \eqref{eq:coupling}-\eqref{eq:conic} are called \emph{coupling constraints}, if they explicitly depend on the follower variables $y$.
Furthermore, 
$q(y)$ is a convex-quadratic function of the form
$q(y) = y'Ry + \linq'y$ with $R=V'V$ and 
$V \in \R^{n_3 \times n_2}$ with $n_3 \leq n_2$,
$\mathcal{K}$ is a
 cross-product of
second-order cones, and $\mathcal{Y}$ is a polyhedron. 
Let $F(x)$ denote the set of feasible solutions of the follower problem for a given $x$, i.e., $F(x) := \left\{ y \in \mathbb Z^{n_2}  : Ax+By \geq f,~ y \in \mathcal{Y}  \right\}$.
A solution $(x,y) \in \mathbb R^n$ is called \emph{bilevel feasible}, if it satisfies all constraints~\eqref{eq:coupling}-\eqref{eq:x}; 
%\red{belongs to the set 
%\[ \mathcal{F} = \{ (x,y) : (x,y) \in 
%is feasible for~\eqref{bilevel} 
otherwise it is called \emph{bilevel infeasible}.
The \IBNP \eqref{bilevel} is called \emph{infeasible} if there is no bilevel-feasible solution.

%\todo[inline]{The following paragraph is not true, this would require continuous leader variables!!}
%\todo{@Jon: has to think about it (explanation: integer to not have continuous variables). {\color{blue} Jon replies: sorry, I don't know what the question/todo is exactly }}
Note that even though the objective function \eqref{eq:objective} is linear, 
%only consider linear objective functions, 
we can actually consider any convex objective function which can be represented as a second-order cone constraint and % $\ell(x,y)$ 
whose optimal value is integer for $(x,y) \in \mathbb{Z}^{n}$ (e.g., a convex-quadratic polynomial with integer coefficients). To do so, we can use an epigraph reformulation to transform it into a problem of the form~\eqref{bilevel}.
%Furthermore, we assume $A$ and $B$ to be integer matrices.
%

Our work considers the \emph{optimistic} case of bilevel optimization. This means that whenever there are multiple optimal solutions for the follower problem~\eqref{eq:IBNP_LL}, the one which is best for the leader is chosen, see, e.g., \cite{loridan1996weak}. We note that already mixed-integer bilevel linear programming (\MIBLOP) is $\Sigma_2^p$-hard~\cite{Lodi-et-al:2014}.

The \emph{value function reformulation} (\VFR) of the bilevel program~\eqref{bilevel} is
\begin{subequations}
	\label{vfr}
	\begin{align}
	&\min  ~c'x + d'y \\%~\ell(x,y) \\
	%&\mbox{subject to:} \nonumber \\
	\st~ &Mx+Ny \geq h \label{ineq:linear}\\
	&\tilde{M}x+\tilde{N}y -  \tilde{h} \in \mathcal{K} \label{ineq:tilde}\\
	& Ax+By \geq f \label{ineq:linking}\\
	&q(y) \leq \Phi(x) \label{eq:inequValueFunction}\\
	& y \in \mathcal{Y} \label{ineq:yinY} \\
	%&x_i, y_j \in \mathbb Z \quad \forall i \in I, j \in J,  
	&(x,y) \in \mathbb Z^n,
	\label{ineq:xyint}
	\end{align}
\end{subequations}
where the so-called \emph{value function} $\Phi(x)$ of the \emph{follower problem}
\begin{equation}
\Phi(x) := {\rm min}\left\{
q(y)  :    y \in F(x) 
%Ax+By \geq f,~  y \in \mathcal{Y}, 
%~ y \in \mathbb Z^{n_2}
%~y_j \in \mathbb Z \quad \forall j \in J
\right\} \label{eq:follower}
\end{equation}
is typically non-convex and non-continuous. Note that {in the optimistic bilevel setting,} the \VFR is equivalent to the original bilevel program~\eqref{bilevel}.
The \emph{high-point relaxation} (\HPR) is obtained when 
dropping~\eqref{eq:inequValueFunction}, i.e., the optimality condition of $y$ 
for the follower problem, from the \VFR~\eqref{vfr}. 
We denote the continuous relaxation (i.e., replacing the integer constraint \eqref{ineq:xyint} with the corresponding variable bound constraints)
of the \HPR as \ccrHPR.

% \ccrHPR, but not feasible for the original bilevel program~\eqref{bilevel}. \todo{\MS{not do this!} Define bilevel feasible first (copy from the Fischetti paper). If it is not bilevel feasible, then it is bilevel infeasible.}

\subsection{Contribution and outline}
%\todo{update}

Since the seminal work of Balas~\cite{BalasDP}, and more intensively in the past three decades, disjunctive cuts (\DCs) have been successfully exploited for solving mixed-integer (nonlinear) programs (MI(N)LPs)~\cite{balas2018disjunctive}.  While there is a plethora of work on using \DCs for MINLPs~\cite{belotti2011disjunctive}, we are not aware of any previous applications of \DCs for solving \IBNPs.  
In this work, we demonstrate how   
\DCs can be used within a branch-and-cut (\BC) algorithm 
%based on the \emph{continuous relaxation} of the \HPR in order 
to solve \eqref{bilevel}.  
%The core of the \BB algorithm is a generalization of  the one recently used  in~\cite{fischetti2016intersection,fischetti2017new,fischetti2018use}   for solving \emph{mixed-integer bilevel linear programs} (\MIBLOPs).
This is the first time that \DCs are used to separate bilevel-infeasible solutions, using a cut-generating procedure based on second-order cone programming (\SOCP).   
Moreover, we also show that our \DCs can be used in a finitely-convergent cutting-plane procedure for 0-1 \IBNPs, where the \HPR is solved to optimality before separating bilevel-infeasible solutions.

In our preliminary study~\cite{Gaar-et-al:2022}, we described the methodological foundations of our approach, based on the assumption of having a single linking constraint in the follower problem. In this paper, we generalize these results for multiple linking constraints (leading to a cut-generating SOCP with multiple disjunctions). We additionally compare DCs derived from non-optimal  versus optimal follower solutions, and show that they are not dominating one another. Moreover, we discuss efficient methods for eliminating redundant disjunctions and normalization procedures for solving the cut-generating SOCP.

%\todo{rewrite this part at the end} 
Our computational study, which is considerably extended compared to~\cite{Gaar-et-al:2022}, is conducted on instances in which the follower minimizes a convex-quadratic objective function, subject to %a 
covering constraints linked with the leader. We consider instances with a single and with multiple linking constraints, and instances with only binary and with integer variables.
We demonstrate that the proposed enhancements of our solution algorithms improve their performance. 
Furthermore, we compare our \BC and cutting-plane approaches with a state-of-the-art solver for \MIBLOPs (which can solve our binary instances after applying linearization in a McCormick fashion) and we show that the latter one is outperformed by our new DC-based approaches. 
%\todo{Correctness to be discussed, and under which conditions}

%\subsection{Outline}

{
Our article is organized as follows. In the remainder of this section, we discuss previous and related work.
In Section~\ref{sec:disj} we describe the derivation of the \DCs. 
Section~\ref{sec:comp_methodology} contains a discussion of computational methodology which allows a successful use of our \DCs. In particular, in Section~\ref{sec:septheory} we demonstrate that our \DCs cut off bilevel-infeasible solutions. In Section~\ref{sec:chooseyhat}, we demonstrate that cuts derived from optimal follower solutions need not dominate cuts derived from non-optimal follower solutions. In Section~\ref{sec:separation}, we describe two different separation strategies, in Section~\ref{sec:removal}, we discuss several approaches to remove redundant disjunctions, and in Section~\ref{sec:normalization}, we address  normalization.
We present a \BC algorithm for solving \eqref{bilevel}  in Section~\ref{sec:bc} and a cutting-plane algorithm for 0-1 \IBNPs in Section~\ref{sec:intcut}.
In Section~\ref{sec:comp}, we present a computational study, together with some implementation details. In Section~\ref{sec:conclusions}, we conclude with an outlook to further work.}

%\subsection{Previous and related work}

\subsection{Literature overview}
%\todo{Update: Ivana}
In recent years, there has been considerable research interest on \BPs. When it comes to solution approaches, a distinction between problems with convex and non-convex follower problem 
%(e.g., due to integrality constraints like in our problem~\eqref{bilevel}) follower 
can be made. For \BPs with a convex follower problem, single-level reformulation techniques based on, e.g., Karush-Kuhn-Tucker optimality conditions or strong duality (see, e.g., \cite{calvete2020algorithms,kleinert2021outer,kleinert2020there}) can be used.
For \MIBLOPs with integrality restrictions on (some of) the follower variables, 
%Due to the difficult nature of these types of problems, all the solution approaches usually make some assumptions on the structure of the problem. The first solution algorithm for \MIBLOPs proposed in the seminal paper~\cite{moore1990mixed} considers problems that have no leader constraints and both the leader and follower variables are mixed-integer and was based on the \HPR. Most solution algorithms for \MIBLOPs, including the ones described below, follow such an approach (they may also consider the linear programming relaxation of the \HPR), and they differ in the way the remove bilevel-infeasible solutions. In~\cite{moore1990mixed} this is done with a \BB algorithm which uses some bilevel specific branching-rules. In~\cite{denegre2009branch} the authors present a branch-and-cut (\BC) framework, which uses bilevel-specific integer no-good cuts. In their problem variant, no continuous variables or leader constraints with follower variables are allowed. This framework was recently extended by including many current developments from literature and allowing for mixed-integer variables in~\cite{tahernejad2020branch}. In~\cite{xu2014exact} a \BB algorithm for \MIBLOPs where all leader variables need to be integral and bounded was proposed. The algorithm uses multiway-branching to deal with bilevel-infeasible solutions. Further algorithms based on multiway-branching are developed in~\cite{liu2021enhanced,wang2017watermelon}. 
state-of-the-art methods are usually based on \BC (see, e.g.,
%\cite{fischetti2017new,fischetti2018use,Tahernejad2020}
\cite{fischetti2017new,fischetti2018use,Tahernejad2020}). Other interesting 
concepts are based on multi-branching, 
see~\cite{wang2017watermelon,xu2014exact}. 

Considerably fewer results are available for nonlinear \BPs, and in particular with integrality restrictions in the follower problem. In~\cite{mitsos2008global}, Mitsos et al.\ propose a general approach for non-convex follower problems %(without explicitly considering integrality restrictions) 
%is proposed %and some theoretical results on attainability of optimal bilevel solutions in this settings are discussed. 
%The approach uses an iterative bounding scheme, 
which solves nonlinear optimization problems to compute upper and lower bounds in an iterative fashion. %Another stream of research in this area 
In a series of papers on the so-called \emph{branch-and-sandwich} approach,  tightened bounds on the optimal value function and on the leader objective-function value are calculated~\cite{KleniatiAdjiman2014a,KleniatiAdjiman2014b,kleniati2015generalization}. 
%Continuous but nonconvex follower problems are considered in~\cite{KleniatiAdjiman2014a,KleniatiAdjiman2014b}, and the approach is extended to the mixed-integer case in~\cite{kleniati2015generalization}. 
%The approach stated in the latter paper is applicable to problems with twice continuously differentiable functions requires bounds on all variables. In this setting, the branch-and-sandwich approach terminates in finite time. 
%Recently, novel bounding schemes for this approach have been published in Paulavicius and Adjiman (2020) and further implementation details can be found in Paulavicius et al. (2020) . Due to the general hardness of the problems under consideration, 
%The computational study in~\cite{kleniati2015generalization} deals with rather small problems with up to 12 variables and 7 constraints.
 A solution algorithm for mixed-\IBNPs proposed in~\cite{lozano2017value} by Lozano and Smith
%under the assumption that the leader variables are all integer, and also the leader-part of the linking constraints is integer is proposed. The algorithm 
approximates the value function by dynamically inserting additional variables and big-M type  constraints. 
%While the approach is proposed for \MIBNPs, it is only evaluated on a competitive scheduling problem and on the \MIBLOP instances from~\cite{xu2014exact}.
Recently, Kleinert et al.~\cite{kleinert2021outer} considered \BPs with a mixed-integer convex-quadratic leader and a continuous convex-quadratic follower problem. The method is based on outer approximation after the problem is reformulated into a single-level one using strong duality and convexification.  
In~\cite{byeon2022benders}, Byeon and Van Hentenryck develop a solution algorithm for \BPs, where the leader problem can be modeled as a mixed-integer \SOCP and the follower problem can be modeled as a \SOCP. The algorithm is based on a dedicated Benders decomposition method.
%After reformulating the problem into a nonconvex single-level one (using a  strong-duality-based transformation) and its convexification, the authors derived two outer approximation approaches, one based on a cutting-plane and the other based on a \BC method. 
In~\cite{weninger2020}, Weninger et al.\ propose a methodology that can tackle any kind of MINLP for the leader which can be handled by an off-the-shelf solver. The mixed-integer follower problem has to be convex, bounded, and satisfy Slater's condition for the continuous variables. This exact method is derived from a previous approach proposed in~\cite{YueGZY19} by Yue et al.\ for finding feasible solutions.
%For  this very challenging class of \MIBNP problems only proof-of-concept computational experiments (using very small instances)   are provided. 
%In~\cite{Cerulli2021}, the authors tackle convex nonlinear problems at the leader and continuous quadratic problems at the follower.
For a more detailed  overview of the recent literature on computational bilevel programming we refer   to~\cite{CerulliThesis,kleinert2021survey,smith2020survey}.

The only existing application of  DCs in the context of bilevel \emph{linear} programming is by Audet et al., \cite{audet} who derive DCs from LP-complementarity conditions. In~\cite{judice}, J{\'{u}}dice et al.\ exploit a similar idea for solving mathematical programs with equilibrium constraints.   

DCs are frequently used for solving MINLPs (see~\cite{balas2018disjunctive} and the many references therein, and for example~\cite{D_Ambrosio_2020,FL2022,MIQCP,MIQCPproj}). 
%In~\cite{Kilinc2014}, K{\i}l{\i}n{\c{c}}-Karzan  and Y{\i}ld{\i}z derive closed-form expressions for inequalities describing the convex hull of a two-term disjunction applied to the second-order cone. 
Concerning the existing literature that includes (computational) studies on DCs for mixed-integer SOCPs, we refer the reader to~\cite{Atamturk-Narayanan:2010,Atamturk-Narayanan:2011,Cezik-Iyengar:2005,Kilinc-Karzan-Steffy:2016,kilincc2015two,Kilinc-Karzan:2016,Lodi-et-al:2020,Modaresi:2016} and further references therein.

\section{Disjunctive cut methodology \label{sec:disj}}
The aim of this section is to derive \DCs for the bilevel program~\eqref{bilevel} with the help of SOCP; so we want to
derive \DCs that separate %integer
bilevel-infeasible solutions
from the convex hull of bilevel-feasible solutions. 
Toward this end, we assume throughout this section that 
we have a second-order conic convex set $\mathcal{P}$, such that 
%the set of feasible solutions of the \VFR is a subset of $\mathcal{P}$, 
% \sout{\blue{all} \red{bilevel-feasible points are contained in $\mathcal{P}$} and such that} \todo{Elli: ?}
$\mathcal{P}$ is a subset of the set of feasible solutions of the \ccrHPR. 
This implies that $\mathcal{P}$ fulfills 
\eqref{ineq:linear}, \eqref{ineq:tilde}, \eqref{ineq:linking} and \eqref{ineq:yinY} and potentially already some \DCs.
Moreover, we assume that $(x^*,y^*)$ is a bilevel-infeasible point in $\mathcal{P}$.
% The point $(x^*,y^*)$ is an \emph{extreme point} of $\mathcal{P}$, if it is not a convex combination of any other two points of $\mathcal{P}$.
%\todo{add $(x^*,y^*) \in \mathcal{P}$?}
%To this end, the \DC methodology requires a definition of the convex set $\mathcal{D}$, to which we will refer as \emph{disjunctive hull}, such that  $(x^*,y^*) \not \in \mathcal{D}$, and for which we will derive a hyperplane (called \emph{disjunctive cut}) separating $(x^*,y^*)$ from $\mathcal{D}$.
%\todo{@Ivana: this set $\mathcal{D}$ is exactly the convex hull of  \eqref{sec:intcut}, right? If it is, then we should explicitly mention it also afterwards. And we should somehow unify the two explanations (later we have $\mathcal{P}$.}

\subsection{Preliminaries}

%\red{\emph{Assumption on the number of linking constraints removed.}}

%For clarity of exposition in what follows, we  consider only one linking constraint of problem~\eqref{bilevel}, i.e., $m_2 = 1$ and thus $A = a'$ and $B = b'$ for some $a \in \Z^{n_1}$, $b \in \Z^{n_2}$ and $f \in \Z$. Note however that our methodology can be generalized for multiple linking constraints leading to one additional disjunction for every additional linking constraint. 

% This assumption is not needed in the integer (and not mixed-integer) case
%\begin{Assumption}
%\label{as:integerLinkingVariables}
%	Only integer variables appear in the linking constraints.
%\end{Assumption}

% This assumption is not needed we decided in a ZOOM discussion on 2021-08-11
%\begin{Assumption}
%	The leader objective function is integer for every feasible~$y$, so 
%	$q(y) \in \mathbb Z$ for every 
%	$ y \in \mathcal{Y}$ with ${y}_j \in \mathbb Z$ for all $j \in J$.
%\end{Assumption}

{Our general assumptions regarding the structure of the \IBNP are given below.}

\begin{Assumption}\label{as:bound1}
All variables are bounded in the \HPR, and %all follower variables are bounded in the follower problem, i.e. 
$\mathcal{Y}$ is bounded.
\end{Assumption}
%\red{This assumption assumes that the set of all bilevel-feasible solutions is finite, and hence, the \HPR is bounded.}
%
%Assumption~\ref{as:bound1} ensures that the \HPR is bounded. 
%and that there exists an optimal solution to the \HPR. 
%\todo{@End: what if bilevel program infeasible?}
We note that in a bilevel-context already for the linear case of \MIBLOPs,
unboundedness of the \ccrHPR does not imply anything for the original problem, 
all three options (infeasible, unbounded, and existence of an optimum) are possible. For more details see, e.g., \cite{fischetti2018use}. 
%This assumption could be weakened to only requiring the integer variables to be bounded by dealing with the unboundedness within a solution algorithm, but for simplifying the presentation we make the stronger assumption.

%We need the following assumption to make sure that there is no bad behaviour (integer x, for which there is no feasible follower solution (i.e. only fractional solutions) -> Then it is unclear, what would happen
%New insight: This is not a problem, we can just continue branching. 
% \begin{Assumption}
% For every $x$, such that there exists a $y$ with $(x,y)$ being feasible for the \ccrHPR, 
% the follower problem~\eqref{eq:follower} is feasible.
% \todo{Probably can be removed. In this form too strong ($D_1$  in Figure 1a. For any $x$ from there, the set $F(x)=\emptyset$), onls needs to hold for integer $x$. Is now stated as condition in Theorem 1.}
% \end{Assumption}

\begin{Assumption}
\label{as:relHPRsolvable}
\ccrHPR has a feasible solution satisfying its nonlinear constraint  \eqref{ineq:tilde}  strictly, and its dual has a  feasible solution. 
\end{Assumption}
Assumption~\ref{as:relHPRsolvable} ensures that we have strong duality for \ccrHPR
and its dual, and so we can solve the \ccrHPR (potentially with added cuts) to
arbitrary accuracy.

%We note that \todo{Do we want to keep this paragraph? NO} due to the fact that we only consider integer variables, a standard assumption (only integer variables appear in the linking constraints) is fulfilled automatically.  A slightly weaker version of this assumption, i.e., no continuous leader variables occurring in the linking constraints, is needed in any case of solving \MIBNPs to avoid attainability problems, see, e.g., \cite{koppe2010parametric,moore1990mixed}. 

%	\item linear objective function $\ell(x,y)$ of the leader
%	\item $\tilde{M}x+\tilde{N}y -  \tilde{h} \in \mathcal{K}$ consists of only 
%	linear constraints or is not present at all
%	\item $y \in \mathcal{Y}$ is only $0 \leq  y \leq 1$
%	\item we have an integer and not a mixed-integer problem
%\end{itemize}

%Based on a bilevel-infeasible (i.e., not feasible for~\eqref{bilevel}) 
%solution 
%$({x}^*,{y}^*)$ of the \todo{continuous relaxation of the} HPR, one can calculate the 
%best follower solution $\hat y$ for $x^*$, so 
%$$\hat y \in \arg \min \left\{
%q(y)  : a'x^*+b'y \geq f,~  y \in \mathcal{Y}, ~y_j \in \mathbb Z \quad \forall j \in J
%\right\}, $$
%for which $q(\hat y)  = \Phi(x^*)$ holds.

\subsection{Deriving disjunctive cuts}
\label{sec:DerivingCuts}

To derive \DCs, we first examine 
%\red{bilevel-free sets. A set is said to be \emph{bilevel-free} if it does not contain any bilevel-feasible solution.} 
%}
bilevel-feasible sets. 
It is easy to see,
and also follows from results of Fischetti et al.~\cite{fischetti2017new}, that 
for any $\hat y \in \mathcal{Y} \cap \mathbb{Z}^{n_2}$
the set 
$$
S(\hat y) := \{ (x,y) : Ax \geq f - B\hat{y},~ q(y) > q(\hat{y})  \}
$$
does not contain any bilevel-feasible solution, as for any $(x,y) 
\in S(\hat y)$ clearly $\hat y$ is a better follower solution for $x$ than $y$. 
Furthermore, due to the integrality of our variables 
%Assumption~\ref{as:integerLinkingVariables} 
and 
%under the integrality assumptions on 
of $A$ and $B$, the  
extended set 
\[
S^+(\hat y) := \{ (x,y) : Ax \geq f - B\hat{y}-1,~ q(y) \geq q(\hat{y})  
\}
\]
does not contain any bilevel-feasible solution in its interior, because any 
bilevel-feasible solution in the interior of $S^+(\hat y)$ 
is in $S(\hat y)$. 
Based on this observation, intersection cuts have been derived in~\cite{fischetti2017new}. However, $S^+(\hat y)$ is not convex in our case, so we turn our attention to \DCs.
For any  $\hat y \in \mathcal{Y}  \cap \mathbb{Z}^{n_2} $, any
bilevel-feasible solution is in the disjunction $\mathcal{D}_0(\hat{y}) \vee \mathcal{D}_1(\hat{y}) \vee \ldots \vee \mathcal{D}_{m_2}(\hat{y})$, where
\begin{align*}
%& & {\color{red}\mbox{\underline{dual mult.}}} \nonumber \\
%&\mathcal{D}_1 : b'y \leq f + 1 - a'\hat{x}, & {\color{red}\sigma \leq 0},
%&
\mathcal{D}_0(\hat{y}) : q(y) \leq q(\hat{y})
\quad \text{ and } \quad
\mathcal{D}_i(\hat{y}) : A^i x \leq f_i - B^i \hat{y}  - 1, \quad i=1,\ldots,m_2. %, & {\color{red}\sigma \leq 0},
%\end{align*}
%and
%\begin{align*} 
%&\mathcal{D}_2 : y'Ry + d'y \leq \hat{y}'R\hat{y} + d'\hat{y}~.
%&
\end{align*}
%Our ultimate goal in order to solve the bilevel program~\eqref{bilevel} is to 
%iteratively add \DCs to the \ccrHPR in order to get rid of bilevel 
%infeasible solutions. This is similar to the algorithm for \MIBLPs presented in~\cite{fischetti2016intersection,fischetti2017new,franceschi2018bilevel}, where \emph{intersection cuts} were used for cutting off bilevel-infeasible solutions. 
%
%We note that while theoretically, our cuts could be used within a cutting-plane algorithm to solve \eqref{bilevel}, in practice, they should be used within a branch-and-cut algorithm, where the branch-and-bound part ensures that the obtained solution is bilevel-feasible, i.e., the cut are used as \emph{valid inequalities}. Due to space reasons,  details about a branch-and-bound algorithm to solve \eqref{bilevel} will only be described in the journal version, we refer to~\cite{fischetti2017new} which discusses a branch-and-bound algorithm to solve \MIBLPs for the main ideas needed.
To find a \DC,
we want to generate valid linear inequalities
for
\begin{equation*}
\left\{(x,y)\in \mathcal{P} : \mathcal{D}_0(\hat{y})\right\}
\vee
\bigvee_{i=1}^{m_2}
\left\{(x,y)\in \mathcal{P} : \mathcal{D}_i(\hat{y})\right\},
\end{equation*}
so in other words we want to find a valid linear inequality that separates the bilevel-infeasible solution $(x^*,y^*)$ from the convex hull of the union of multiple disjunctions, namely from
\begin{equation}
%\label{eq:Dyhat}
\mathcal{D}(\hat{y},\mathcal{P}) := 
\conv \left(
\left\{(x,y)\in \mathcal{P} : \mathcal{D}_0(\hat{y})\right\}
\cup
\left(\bigcup_{i=1}^{m_2}
\left\{(x,y)\in \mathcal{P} : \mathcal{D}_i(\hat{y})\right\}
\right)\right).
\end{equation}

Toward this end, we first derive a formulation of $\mathcal{P}$.
If we have already generated
some \DCs of the form $\alpha'x +\beta'y \geq \tau$, then we group them as
$ \mathcal{A}x + \mathcal{B} y \ge \mathcal{T}$. 
We take these cuts, together with $Mx+Ny \geq h$ and $Ax+By \geq f$ and also 
$y \in \mathcal{Y}$, which can be represented as $\mathcal{C} y \ge \mathcal{U}$,
and we bundle them all together as
\begin{align}
%& & {\color{red}\mbox{\underline{dual mult.}}} \nonumber \\
&\bar{M}x + \bar{N}y \geq \bar{h}, %, & {\color{red}\bar{\pi}_k \geq 
%0},
\label{ineq:bar}
\end{align}
such that $\mathcal{P}$ is represented by \eqref{ineq:bar} and \eqref{ineq:tilde}, and where
\begin{align*}
\bar M := \begin{pmatrix}
M \\
A \\
\mathcal{A} \\
0 
\end{pmatrix}%_{\ell \times n_1}
,
\qquad
\bar N := \begin{pmatrix}
N \\
B \\
\mathcal{B} \\
\mathcal{C}
\end{pmatrix}%_{\ell \times n_2}
,
\qquad
\bar h := \begin{pmatrix}
h \\
f \\
\mathcal{T} \\
\mathcal{U}
\end{pmatrix}%_{\ell \times 1}
.
\end{align*}

The representation of $\mathcal{D}_i(\hat{y})$, $i=1,\dots,m_2$, is straightforward. 
It is convenient to write $\mathcal{D}_0(\hat{y})$ in 
SOCP-form using a standard technique.
Indeed, $\mathcal{D}_0(\hat{y})$ is equivalent to the standard 
second-order (Lorentz) cone constraint
$z^0 \geq \left\| (z^1,z^2) \right\|$ with 
\begin{align*}
%&z^0 := \frac{1-\left(d'y- \hat{y}'R\hat{y} - d'\hat{y}\right)}{2},\\
%&
z^0 := \frac{1-\left(\linq'y- q(\hat{y}) \right)}{2},%\\
%&
\qquad 
z^1 := Vy, 
%\qquad \text{ and } 
\qquad
%\\
%&z^2 := \frac{1+\left(d'y- \hat{y}'R\hat{y} - d'\hat{y}\right)}{2}.
%&
z^2: = \frac{1+\left(\linq'y- q(\hat{y})\right)}{2}.
\end{align*}
Because $z^0$, $z^1$ and $z^2$ are linear in $y$, we can as well write it in the form
\begin{align}
\label{con:newAtilde}
%& & {\color{red}\mbox{\underline{dual mult.}}} \nonumber \\
& \tilde{D}y - \tilde{c} \in \mathcal{Q}, % , & {\color{red}\rho  \in 
%\mathcal{Q^*}},
\end{align}
where $\mathcal{Q}$ denotes a standard second-order cone,  which is self dual, and  
%In particular, the constraint $\tilde{A}y - \tilde{b} \in \mathcal{Q}$ is 
%instantiated by taking
\[
\tilde{D} := \left(
\begin{array}{c}
-\frac{1}{2}\linq' \\[3pt]
V \\[3pt]
\frac{1}{2}\linq'
\end{array}
\right)
\qquad \text{ and } \qquad
\tilde{c} := \left(
\begin{array}{c}
%                 \frac{-1-\hat{y}'R\hat{y}-d'\hat{y}}{2} \\[3pt]
\frac{-1-q(\hat{y})}{2} \\[3pt]
%\vec{0} 
0
\\[3pt]
%                  \frac{-1+\hat{y}'R\hat{y}+d'\hat{y}}{2}
\frac{-1+q(\hat{y})}{2}
\end{array}
\right).
\]

We employ a vector $\sigma$ of 
 dual multipliers for the linear constraints representing $\mathcal{D}_1(\hat{y}),\ldots,\mathcal{D}_{m_2}(\hat{y})$. Moreover, we 
 employ a vector $\rho \in \mathcal{Q^*}$  of dual multipliers for 
 the constraint~\eqref{con:newAtilde}, representing~$\mathcal{D}_0(\hat{y})$. 
 Furthermore, we employ vectors
 $\bar{\pi}_i$, $i=0,\ldots,m_2$,
 of dual multipliers for the constraints \eqref{ineq:bar},
 and we employ vectors
 $\tilde{\pi}_i$, $i=0,\ldots,m_2$,
 of dual multipliers for the constraints \eqref{ineq:tilde}, both together representing $\mathcal{P}$.
%We do not have a convenient representation of $\Phi(x)$, 
%so we iteratively build up a conic continuous relaxation  $\mathcal{P}$
%(which includes 
%\eqref{ineq:linear},  \eqref{ineq:tilde}, \eqref{ineq:linking} and \eqref{ineq:yinY}
%and all ``disjunctive cuts'' derived so far)
%of the convex hull of the feasible region of the \VFR.
%In the beginning $\mathcal{P}$ represents \ccrHPR, and then we iteratively add %\DCs.
%We do this based on a sequence of bilevel-infeasible solutions $(x^*,y^*)$, which are the current optimal solutions when optimizing \eqref{eq:objective} over $\mathcal{P}$. These optimal solutions can be obtained because of Assumption~\ref{as:bound1}.
Then
every $(\alpha,\beta,\tau)$ corresponding to
a valid linear inequality $\alpha'x +\beta'y \geq \tau$
for $\mathcal{D}(\hat{y},\mathcal{P})$ corresponds to a solution of
% has the extended-variable representation \todo{@Jon: citation?
% {\color{blue} \scriptsize Jon replies: I don't know if I can find this explicitly written anywhere.
% It really is just an explicit application of convex duality,
% but then the trick (basically well-known) of combining the duals for several convex sets.
% But here is the reasoning, which maybe we should state" ``Every closed convex set C is precisely the solution set of the linear inequalities given by its support 
% function (see Section 13 of Rockafellar's book). 
% And every valid inequality for C is dominated by a linear inequality given by its support 
% function. So, when we have a finite number of closed convex sets, the valid inequalities for the 
% convex closure of these is every inequality that dominates the linear inequalities given by the
% support functions of all them.''
% @Jon: rewrite
% }
% }
\begin{subequations}
\begin{align}
%&\alpha' = \bar{\pi}_1'\bar{M} + \tilde{\pi}_1'\tilde{M},\label{eq1} \\
&\alpha' = \bar{\pi}_i'\bar{M} + \tilde{\pi}_i'\tilde{M} +\sigma_i 
A^i && \forall i =1,\ldots, m_2\label{eq1} \\
&\alpha' = \bar{\pi}_0'\bar{M} + \tilde{\pi}_0'\tilde{M} \label{eq2} \\
%&\beta' = \bar{\pi}_1'\bar{N} + \tilde{\pi}_1'\tilde{N}  +  \sigma \tilde{b}', 
%\label{eq3} \\
&\beta' = \bar{\pi}_i'\bar{N} + \tilde{\pi}_i'\tilde{N}   && \forall i =1,\ldots, m_2 \label{eq3} \\
&\beta' = \bar{\pi}_0'\bar{N} + \tilde{\pi}_0'\tilde{N}  +  \rho' \tilde{D} 
\label{eq4} \\
%&\tau \leq \bar{\pi}_1'\bar{h} + \tilde{\pi}_1'\tilde{h} + 
%\sigma(f+1-a'\hat{y}), \label{eq5} \\
&\tau \leq \bar{\pi}_i'\bar{h} + \tilde{\pi}_i'\tilde{h} + 
\sigma_i(f_i-1-B^i\hat{y}) && \forall i =1,\ldots, m_2 \label{eq5} \\
&\tau \leq \bar{\pi}_0'\bar{h} + \tilde{\pi}_0'\tilde{h} + \rho' \tilde{c} 
\label{eq6} \\
% &\bar{\pi}_1 \geq 0,~ \bar{\pi}_2 \geq 0,~
% \tilde{\pi}_1  \in \mathcal{K^*},~ \tilde{\pi}_2  \in \mathcal{K^*},~
&\sigma\leq 0,~ \rho \in \mathcal{Q^*},~\bar{\pi}_i \geq 0,~ \tilde{\pi}_i \in \mathcal{K^*} && \forall i=0,\ldots,m_2, \label{eq7}
\end{align}
\end{subequations}
 where $\mathcal{K}^*$ and $\mathcal{Q^*}$ are the dual cones of  $\mathcal{K}$ and $\mathcal{Q}$, respectively
 (see, e.g., Balas~\cite[Theorem 1.2]{balas2018disjunctive}).

To attempt to generate a valid inequality for $\mathcal{D}(\hat{y},\mathcal{P})$ that is violated by
the bilevel-infeasible solution
$(x^*,y^*)$, we solve
\begin{align*}
%&\max\ \tau - \alpha'\hat{x} +\beta'\hat{y}, \tag{CG-SOCP} \label{CG-SOCP}\\
&\max\ \tau - \alpha'x^* -\beta'y^* \tag{CG-SOCP} \label{CG-SOCP}\\
\st &~\eqref{eq1}\mbox{--}\eqref{eq7}.
\end{align*}
%which is possible because of Assumption~\ref{as:CGSOCPsolvable}.
A positive objective value for a feasible
$(\alpha,\beta,\tau)$ corresponds to
a valid linear inequality $\alpha'x +\beta'y \geq \tau$
for $\mathcal{D}(\hat{y},\mathcal{P})$
violated by
%$(\hat{x},\hat{y})$.
$(x^*,y^*)$, i.e., the inequality gives a \DC separating $(x^*,y^*)$ from  $\mathcal{D}(\hat{y},\mathcal{P})$.
Finally, we need to deal with the fact that the feasible region of \eqref{CG-SOCP}
is a cone. We will take care of this in the usual manner, by  including a normalization 
constraint; see Section~\ref{sec:normalization}.

\section{Computational methodology for our disjunctive cuts} \label{sec:comp_methodology}
%\todo{@Jon: make title nicer}
 %\section{Algorithmic Considerations}
 %\todo{update}
 In this section we discuss theory and methodology of our proposed DCs. 
 %Theorem~\ref{tm:getCut} provides a theoretical foundation for the usefulness of our DCs for designing solution algorithms for \IBNPs. Since the cuts depend on the choice 
 %of the point $\hat y$ through the set $\mathcal{D}(\hat{y},\mathcal{P})$,
 %we discuss how this choice affects the obtained \DC. 
 %We also propose several procedures for eliminating redundant disjunctions,  and discuss various normalization strategies for \eqref{CG-SOCP}.

\subsection{Separation theory \label{sec:septheory}}
%\todo{@Jon: make title nicer}
To be able to derive \DCs we make the following additional assumption.
\begin{Assumption}
\label{as:CGSOCPsolvable}
%\todo{\scriptsize @Jon: Is this assumption sufficient?
%{\color{blue}
%First, let me say that we need to add the normalization \emph{before} we introduce this assumption.
%Then 
%what we need to assume is only an interior solution 
%for the dual of \eqref{CG-SOCP}. Interior only in the sense of interior to the cone constraints of the dual.
%As for the primal, we only need ordinary feasibility, but that comes for free: 0 is always feasible to
%\eqref{CG-SOCP}
%}
%}
The dual of \eqref{CG-SOCP} has a feasible solution in its interior, 
and we have an exact solver for \eqref{CG-SOCP}.
\end{Assumption}

%\MS{We note that depending on the normalization chosen for solving \eqref{CG-SOCP}, Assumption~\ref{as:CGSOCPsolvable} may not hold. This is the case if all the disjunctions are empty. However, in this case, }

The following theorem allows us to use \DCs in our solution methods.
%, and which we will not prove for the sake of brevity.
 
% \todo{\scriptsize @Jon: Do you agree that the following theorem is correct?}
% \todo{\scriptsize @Jon: Can we guarantee that by optimizing the linear objective function over $\mathcal{P}$ we always attain the optimal solution at an extreme point? {\color{blue} Jon replies: If, by an extreme point you mean that it is in the intersection
% of $\mathcal{P}$ with a supporting hyperplane of  $\mathcal{P}$, then yes, if we assume that $\mathcal{P}$ 
% is a closed set. }
% }
 \begin{theorem}\label{tm:getCut}
%  \todo{Elli: Don't we use this theorem actually in the case where all bilevel *optimal* (and not all bilevel) points are contained in P?}
Let $\mathcal{P}$ be a second-order conic convex set, such that 
%the set of feasible solutions of the \VFR is a subset of 
%\sout{\red{\blue{all}
%\todo{Elli: we want all here, right? I: Yes} 
%bilevel-feasible points are %contained in $\mathcal{P}$,} and %such that }
$\mathcal{P}$ is a subset of the set of feasible solutions of the \ccrHPR.
 Let $(x^*,y^*)$ be a bilevel-infeasible extreme point of $\mathcal{P}$. 
 %Let $\hat y$ be a feasible solution to the follower problem for $x=x^*$ (i.e., $\hat y \in \mathcal{Y} \cap \mathbb Z^{n_2}$  and $a'x^*+b' \hat{y} \geq f$) such that $q(\hat{y}) < q(y^*)$.
 If $\hat y$ is a feasible solution to the follower problem for $x=x^*$, i.e., $\hat y \in F(x^*)$,
 %$\hat y \in \mathcal{Y} \cap \mathbb Z^{n_2}$  and $a'x^*+b' \hat{y} \geq f$) 
 such that $q(\hat{y}) < q(y^*)$,
 then there is a \DC that separates $(x^*,y^*)$ from  $\mathcal{D}(\hat{y},\mathcal{P})$ and it can be obtained by solving \eqref{CG-SOCP}.
 \end{theorem}
 \begin{proof}
 Assume that there is no cut that separates $(x^*,y^*)$ from  $\mathcal{D}(\hat{y},\mathcal{P})$, then $(x^*,y^*)$ is in $\mathcal{D}(\hat{y},\mathcal{P})$. 
 However, due to the definition of $\hat{y}$, 
 the point $(x^*,y^*)$ does not fulfill $\mathcal{D}_i(\hat{y})$ for any $i = 0 , \dots, m_2$.
 Therefore, in order to be in $\mathcal{D}(\hat{y},\mathcal{P})$, the point $(x^*,y^*)$ must be a convex combination of some points $(x_i,y_i)$ for $i = 0, \dots, m_2$, such that each $(x_i,y_i)$ is in $\mathcal{P}\cap \mathcal{D}_i(\hat{y})$, and such that at least two coefficients of the convex combination are larger than zero. This is not possible due to the fact that $(x^*,y^*)$ is an extreme point of $\mathcal{P}$.
 Thus, there is a cut that separates $(x^*,y^*)$ from  $\mathcal{D}(\hat{y},\mathcal{P})$. By construction of \eqref{CG-SOCP} and due to   Assumption~\ref{as:CGSOCPsolvable}, we can use \eqref{CG-SOCP}  to find it.
 \ifArXiV %---------------------------------------------------------
\else %--------------------------------------------------------------
\qed
\fi %--------------------------------------------------------------
 \end{proof}
 
%Comment: In the theorem we need 
% -) $\hat y \in \mathcal{Y}$ with $\hat{y}_j \in \mathbb Z$ for all $j \in J$ to have a bilevel free set $S^+(\hat{y})$,
% -) $a' x^*+b' \hat{y} \geq f$ such that $(x^*,y^*)$ is not in $\mathcal{D}_1(\hat{y})$, and 
% -) $q(\hat{y}) < q(y^*)$ such that that $(x^*,y^*)$ is not in $\mathcal{D}_2(\hat{y})$

 Note that there are two reasons why a feasible \ccrHPR solution $(x^*,y^*)$ 
 %of a conic continuous relaxation of the \VFR (potentially with already added \DCs) that contains all constraints of \ccrHPR 
 is bilevel infeasible: it is not integer or $y^*$ is not an optimal follower solution for $x^*$.
 Thus, in the case that $(x^*,y^*)$ is integer, there is a better follower solution $\tilde{y}$ for $x^*$. Then Theorem~\ref{tm:getCut} with  
 $\hat{y} = \tilde{y}$ implies that $(x^*,y^*)$ can be separated
 from  $\mathcal{D}(\hat{y},\mathcal{P})$.
% from $\mathcal{P}$ 
% with a \DC. 
 We present solution methods based on this observation in Section~\ref{sec:solution}.
In the case that $(x^*,y^*)$ does not fulfill the integer constraints~\eqref{ineq:xyint}, we   distinguish the following situations:
\begin{itemize}
    \item  $F(x^*)\neq \emptyset$ and there is a better feasible follower solution $\tilde{y}$ for $x^*$, so one could still use a \DC to eliminate $(x^*,y^*)$ due to Theorem~\ref{tm:getCut} with $\hat{y} = \tilde{y}$. %\todo{can fractional $(x^*,y^*)$ be on he boundary of $S^+(\hat y)$, so that potentially we cannot cut it off?}
    \item $F(x^*)\neq \emptyset$ and all $\tilde{y}$ that are feasible for the follower problem $\Phi(x^*)$ have worse (or same) follower objective-function value than $y^*$, so there is no $\tilde{y}$ that we can choose as $\hat{y}$ in Theorem~\ref{tm:getCut}. 
    \item $F(x^*)=\emptyset$, i.e., the follower problem is infeasible for the given fractional point $(x^*,y^*)$. 
    \end{itemize}
    In the latter two cases, the point $(x^*,y^*)$ cannot be cut off using a DC, however we will see below that such points can be discarded using standard integer-programming techniques.  
    Hence, this potential failure to separate a $(x^*,y^*)$ not fulfilling the integer constraints does not affect our solution algorithms.

%  \begin{subequations}
% 	\label{vfr}
% 	\begin{align}
% 	&\min  ~c'x + d'y \\%~\ell(x,y) \\
% 	%&\mbox{subject to:} \nonumber \\
% 	\st~ &Mx+Ny \geq h \label{ineq:linear}\\
% 	&\tilde{M}x+\tilde{N}y -  \tilde{h} \in \mathcal{K} \label{ineq:tilde}\\
% 	& Ax+By \geq f \label{ineq:linking}\\
% 	&q(y) \leq \Phi(x) \label{eq:inequValueFunction}\\
% 	& y \in \mathcal{Y} \label{ineq:yinY} \\
% 	%&x_i, y_j \in \mathbb Z \quad \forall i \in I, j \in J,  
% 	&(x,y) \in \mathbb Z^n,
% 	\label{ineq:xyint}
% 	\end{align}
% \end{subequations}
% where the so-called \emp

\subsection{Choosing the point \texorpdfstring{$\hat y$}{y hat} to separate \label{sec:chooseyhat}}

For a given \DC $\alpha'x +\beta' y \geq \tau$, we say that the \DC is \emph{dominated} if there exists another \DC $\bar \alpha'x +\bar \beta' y \geq \bar \tau$ such that 
$\{ (x,y) \in \mathcal P: \bar \alpha'x +\bar \beta' y \geq \bar \tau\} \subset \{ (x,y) \in \mathcal P: \alpha'x + \beta' y  \geq \tau\}$, otherwise the \DC is called  \emph{non-dominated}. 
%\todo{too strong, better to focus on pairwise domination \MS{this next definition is pairwise domination}} 
If 
$ \{ (x,y) \in \mathcal P: \alpha'x + \beta' y \geq \tau\} \subset \{ (x,y) \in \mathcal P: \bar \alpha'x +\bar \beta' y \geq \bar \tau\}$, we say that the \DC $\alpha'x +\beta' y \geq \tau$
is \emph{dominating} the \DC $\bar \alpha'x +\bar \beta' y \geq \bar \tau$, otherwise the \DC $\alpha'x +\beta' y \geq \tau$ is \emph{not dominating} the \DC $\bar \alpha'x +\bar \beta' y \geq \bar \tau$. 
The following result establishes that for a given $x^*$, a \DC derived from a feasible but non-optimal follower solution does not have to be dominated by a \DC derived from an optimal follower solution. 

%This result justifies our choice of exploring the option \texttt{G} (greedy separation of DCs) in Algorithm~\ref{alg:separation}.

 \begin{theorem} \label{th:nondominance}
%Let $\hat y, \hat y_1 \in \mathcal{Y}$ be two points from the follower  polyhedron, such that there exists $x^* \in \red{\mathcal{X}}$\footnote{define $\mathcal{X}$ or find another workaround} 
%\red{
%Let $F(x)$ denote the set of feasible solutions of the follower problem for a given $x$, i.e., $F(x) = \left\{ y \in \mathbb Z^{n_2}  : Ax+By \geq f,~ y \in \mathcal{Y}  \right\}$.} 
Let $(x^*,y^*)\in \mathcal{P}$ be a bilevel-infeasible solution such that $\Omega(x^*) \neq \emptyset$
 and such that $F(x^*) \setminus \Omega(x^*) \neq \emptyset $. 
%and let $(x^*,\hat y_1), (x^*,\hat y_2)$ such that 
%\todo{why do we need $(x^*,\hat y_1) \in \mathcal P, (x^*,\hat y_2) \in \mathcal P$?} \sloppy
Let $\hat y_1 \in \Omega(x^*)$ and $\hat y_2 \in 
F(x^*) \setminus \Omega(x^*)$. Then, there exist instances where the two \DCs, one derived from $\mathcal{D}(\hat{y}_1,\mathcal{P})$ 
and the other derived from $\mathcal{D}(\hat{y}_2,\mathcal{P})$, do not dominate one another.  
% \todo[inline]{Actually, this theorem is very sloppy written: First we write about $\mathcal{P}$ (which means that we have a concrete instance), and then we say something about existence of an instance. To be mathematically correct, we would have to write something like
%"A \DC derived from an optimal solution to the follower problem does not necessarily dominate all \DCs derived from feasible, but not optimal follower solutions."}
 \end{theorem} 
 \begin{proof}
 To prove this result, we consider an adaptation of the famous example from~\cite{moore1990mixed}, namely
\begin{align*}
  \min_{x\in \Z} \{ x-y : y \in \Omega(x) \}, %\label{eq:objMB}
  \end{align*}
  where $\Omega(x)$ is the set of optimal solutions of the problem  
  \begin{subequations}
\begin{align} \min_{y\in \Z}  \{  y^2  :~ 25x -20y &\ge -30 \label{eq:gMB1}\\
     -x -2 y &\ge -10 \label{eq:gMB2}\\
      -2 x +  y &\ge -4 \label{eq:gMB3}\\
      2 x + 10y &\ge 15 \ \},\label{eq:gMB4}
%%  &\}  
\end{align} 
\end{subequations} 
and $\mathcal{P}$ is the set of feasible solutions to the linear constraints \eqref{eq:gMB1} - \eqref{eq:gMB4}.

For $x^*=2$, we have $\hat y_1 = 2 \in \Omega(x^*)$ and $\hat y_2=3 \in F(x^*) \setminus \Omega(x^*)$. The disjunctions associated with $\hat y_1 = 2$ are
\begin{align*}
&\mathcal{D}_0(\hat{y}_1): y^2 \le 4,  
&&\mathcal{D}_1(\hat{y_1}): x \le 9/25, 
&&\mathcal{D}_2(\hat{y_1}): x \ge 7,  \\
& &&\mathcal{D}_3(\hat{y_1}): x \ge 7/2,
&&\mathcal{D}_4(\hat{y_1}): x \le -3.
\end{align*}
Note that both $\mathcal P \cap \mathcal{D}_2(\hat{y_1})$ and
$\mathcal P \cap \mathcal{D}_4(\hat{y_1})$ are empty.
When \eqref{CG-SOCP} is solved using the cut-coefficient normalization with $1$-norm (see Section~\ref{sec:normalization} for more details on normalization) for the solution $x^*=2$ and $y^*=4$, the \DC obtained is $-1.25x+3.1y\leq 5.7$.

Similarly, the disjunctions associated with $\hat y_2 = 3$ are
\begin{align*}
&\mathcal{D}_0(\hat{y_2}): y^2 \le 9,  
&&\mathcal{D}_1(\hat{y_2}): x \le 29/25, 
&&\mathcal{D}_2(\hat{y_2}): x \ge 5, \\
& &&\mathcal{D}_3(\hat{y_2}): x \ge 4, 
&&\mathcal{D}_4(\hat{y_2}): x \le -8.
\end{align*}
Note that $\mathcal P \cap \mathcal{D}_i(\hat{y_2})$ is empty for all $i\in\{2,3,4\}$.
The corresponding \DC obtained by using cut-coefficient normalization with $1$-norm and the solution $x^*=2$ and $y^*=4$ is $y\leq 3$. Figure~\ref{fig:nondominance} illustrates the two cuts, neither of which dominates the other.
\ifArXiV %---------------------------------------------------------
\else %--------------------------------------------------------------
\qed
\fi %--------------------------------------------------------------
\end{proof}

\begin{figure}[tbh]
    \centering
    \begin{subfigure}{0.48\textwidth}
    \includegraphics[width=1.0\textwidth]{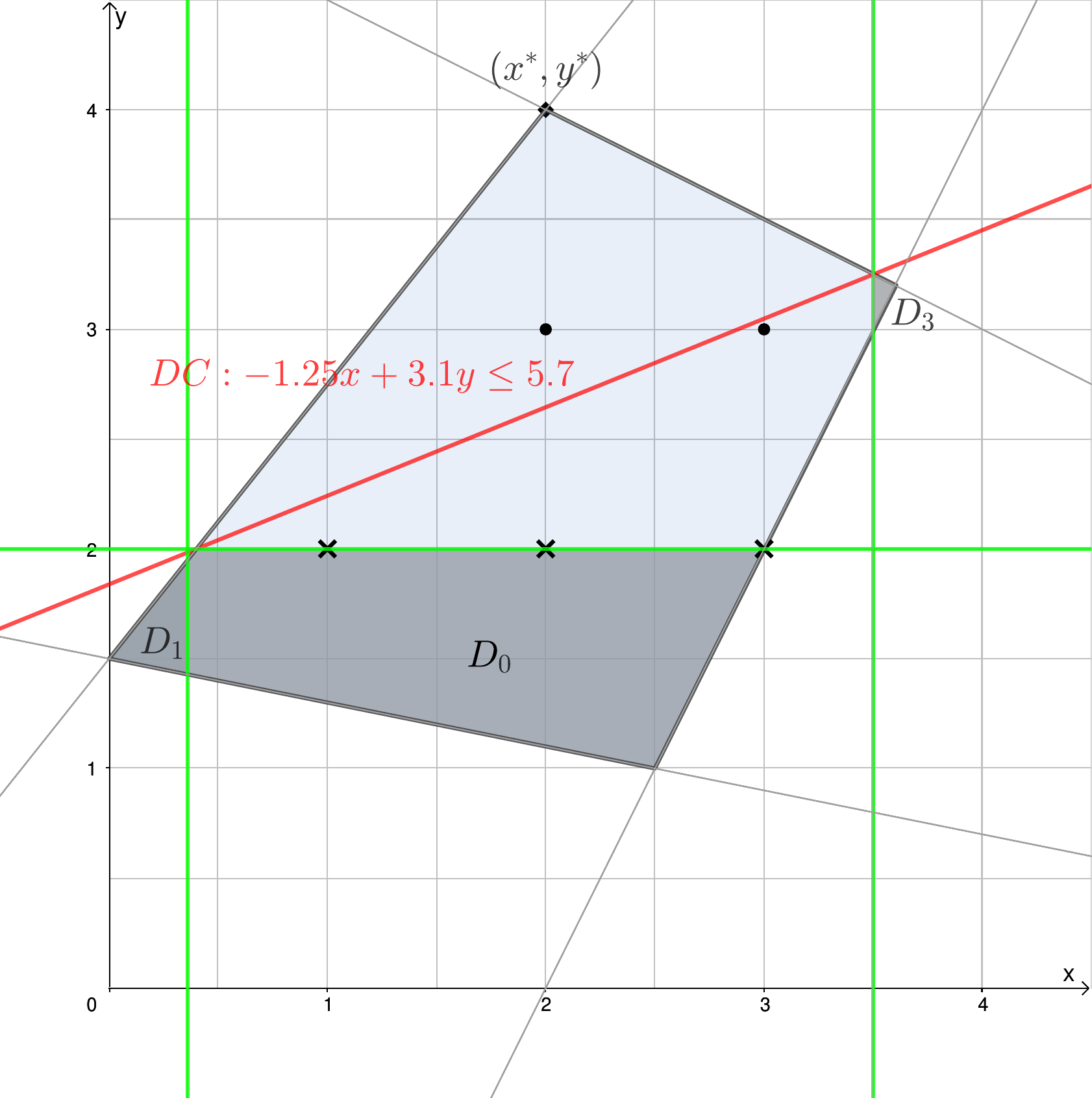}
    \caption{A \DC derived from an optimal follower solution $\hat y_1=2 \in \Omega(x^*)$.\newline\label{fig:nd1}}
    \end{subfigure}\quad
        \begin{subfigure}{0.48\textwidth}
     \includegraphics[width=1.0\textwidth]{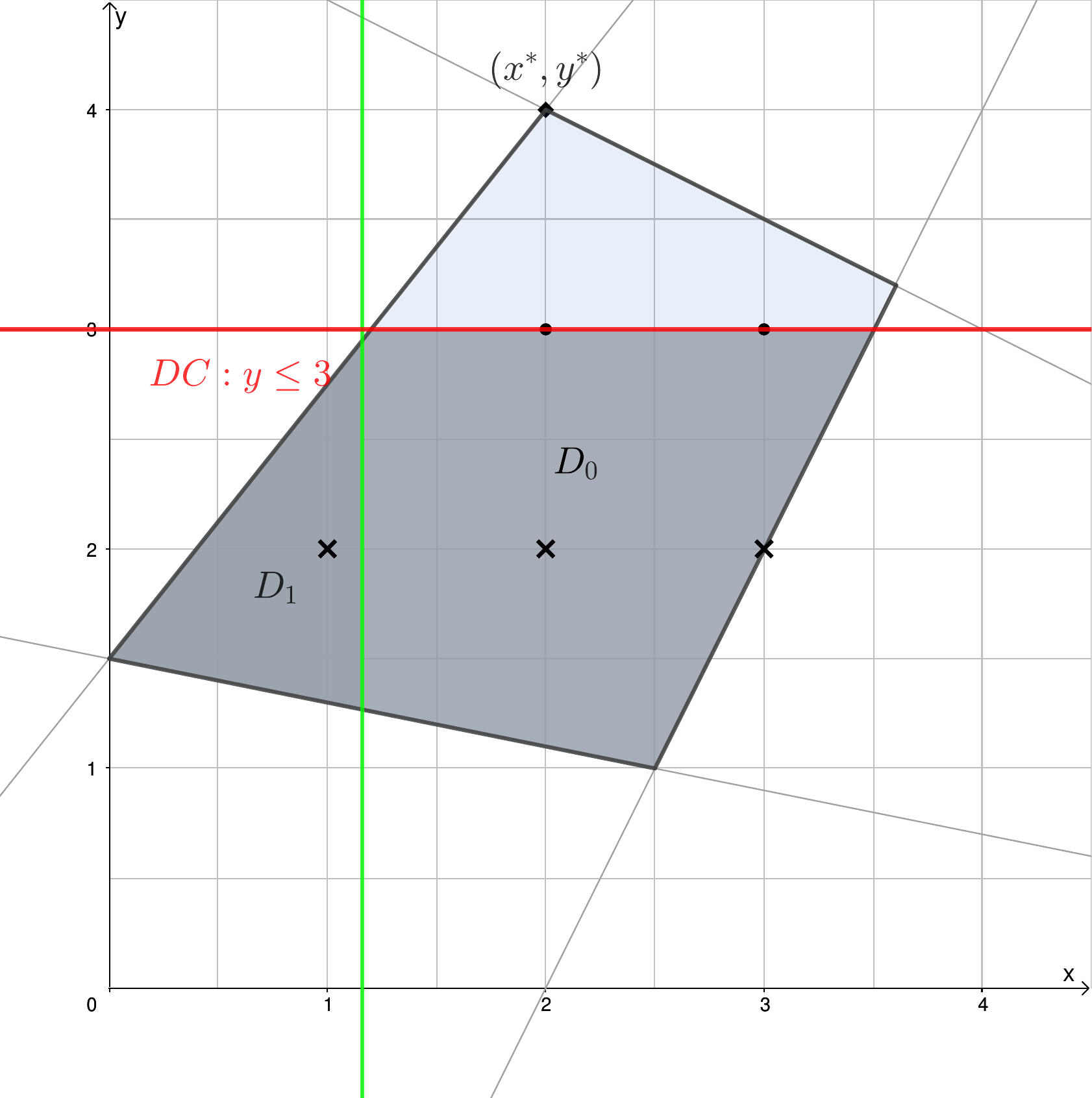}
     \caption{A \DC derived from a feasible but non-optimal follower solution $\hat y_2=3 \in F(x^*)\setminus \Omega(x^*)$. \label{fig:nd2}}
        \end{subfigure}
    \caption{An example illustrating two \DCs which do not dominate one another.}
    \label{fig:nondominance}
\end{figure}
Theorem~\ref{th:nondominance} indicates that multiple \DCs which do not dominate one another could be derived when separating a bilevel-infeasible point $(x^*,y^*)$. Moreover, it also means that we do not need to solve the follower problem to optimality in order to generate a (potentially) non-dominated \DC. This is exploited in one of the separation procedures described in the next section.

%Indeed, its is worth investigating the impact of multiple separation strategies of disjunctive cuts on the quality of the bounds at the root node of the branching tree. We therefore also consider a third option, \texttt{exhaustive search}, in which we enumerate all possible feasible follower solutions for the given choice of $x^*$, and insert the corresponding violated cuts derived from them. Clearly, this is a computationally demanding approach, in which potentially an exponential number of cuts can be generated... TO BE COMPLETED, AND MAYBE WE MOVE IT IN THE COMPUTATIONAL PART  

\subsection{Separation procedures}
\label{sec:separation}
%\todo{should be in line with separation implementation details}
%Note that \eqref{CG-SOCP} gives a valid (but not necessarily violated) cut for every potential $\hat{y}$, so 
%for any $\hat y \in \mathcal{Y} \cap \mathbb{Z}^{n_2}$. We will now discuss potential choices of $\hat y$.

We now turn our attention to describing how to computationally separate our \DCs for a solution $(x^*,y^*) \in \mathcal{P}$.
%to \ccrHPR, potentially with already added \DCs. 
Note that we do not necessarily need the optimal solution of the follower problem~\eqref{eq:follower} for $x=x^*$ to be able to cut off a bilevel-infeasible solution $(x^*,y^*)$, as any $\hat y$ that is feasible for the follower problem  with $q(\hat y)< q(y^*)$ gives a violated \DC as described in Theorem~\ref{tm:getCut}. Motivated by the result of Theorem~\ref{th:nondominance}, we implement two different strategies for separation which are described in Algorithm~\ref{alg:separation}. 

In the first one, denoted as \texttt{O}, we solve the follower problem to optimality, and use the optimal $\hat y$ in \eqref{CG-SOCP}. In the second strategy, denoted as \texttt{G}, for each feasible integer follower solution $\hat y$ with a better objective value than $q(y^*)$ obtained during solving the follower problem, we try 
to solve \eqref{CG-SOCP}. The procedure returns the first-found significantly-violated cut, i.e., it finds a \DC greedily. A cut $\alpha'x +\beta'y \geq \tau$ is considered to be \emph{significantly violated} by $(x^*,y^*)$ if $\tau - \alpha'x^* -\beta'y^*> \varepsilon$ for some $\varepsilon>0$. 
%\todo{value epsilon}

%The following explanation could be removed. 
If $(x^*,y^*)$ is a bilevel-infeasible solution satisfying integrality constraints, Algorithm~\ref{alg:separation} returns a violated cut with both strategies. 
Otherwise, i.e., if $(x^*,y^*)$ is not integer, a cut may not be obtained for the reasons discussed after Theorem \ref{tm:getCut}.

%, because it is possible that there is no feasible $\hat y$ for the follower problem with $q(\hat y)< q(y^*)$. 

%The separation of fractional solutions is not needed for correctness, but can be implemented to enhance(?) the branch-and-cut procedure. In the settings that we separate fractional node solutions $(x^*,y^*)$, similar to integer separation we either solve the follower problem optimally, or solve it until a violated cut is found. 

\SetKwRepeat{Do}{do}{while}

\begin{algorithm}[tbh]
\LinesNumbered
\SetKwInOut{Input}{Input}\SetKwInOut{Output}{Output}
\Input{A feasible \ccrHPR solution $(x^*, y^*)$, a separation \emph{strategy} \texttt{O} or \texttt{G}, a set $\mathcal{P}$} 
\Output{A significantly violated \DC or nothing} 
\BlankLine 
{
\While{the follower problem is being solved for $x=x^*$ by an enumeration-based method}
{
\For{each feasible integer $\hat{y}$ with $q(\hat{y})<q(y^*)$}
{
\If{$\textit{strategy}=\texttt{G}$ or ($\textit{strategy}=\texttt{O}$ and $\hat{y}$ is optimal)}
{
solve (CG-SOCP) for $(x^*, y^*)$, $\hat{y}$ and $\mathcal{P}$\;
\If{$\tau - \alpha'x^* -\beta'y^*>\varepsilon$ }
{
\Return{$\alpha'x +\beta'y \geq \tau$}\;
}
}}}}
% \Return{WHAT\;}
\caption{\texttt{separation}\label{alg:separation}}
\end{algorithm}

\iffalse

\begin{algorithm}[h!tb]
\LinesNumbered
\SetKwInOut{Input}{Input}\SetKwInOut{Output}{Output}
\Input{A feasible \ccrHPR solution $(x^*, y^*)$} 
\Output{A significantly violated disjunctive cut or nothing} 
\BlankLine 
{
solve the follower problem to optimality for $x=x^*$, let $\hat{y}$ be the obtained optimal solution\;
\If{$q(\hat{y})<q(y^*)$}
{
solve (CG-SOCP) for $(x^*, y^*)$ and $\hat{y}$\;
\If{$\tau - \alpha'x^* -\beta'y^*> \varepsilon$}
{
\Return{$\alpha'x +\beta'y \geq \tau$}\;
}
}}
% \Return{WHAT\;}
\caption{\texttt{separation-O}\label{alg:separation-O}}
\end{algorithm} 

\begin{algorithm}[h!tb]
\LinesNumbered
\SetKwInOut{Input}{Input}\SetKwInOut{Output}{Output}
\Input{A feasible \ccrHPR solution $(x^*, y^*)$} 
\Output{A significantly violated disjunctive cut or nothing} 
\BlankLine 
{
\While{the follower problem is being solved for $x=x^*$ by an enumeration based method}
{
\For{each feasible integer $\hat{y}$ with $q(\hat{y})<q(y^*)$}
{
solve (CG-SOCP) for $(x^*, y^*)$ and $\hat{y}$\;
\If{$\tau - \alpha'x^* -\beta'y^*>\varepsilon$ }
{
\Return{$\alpha'x +\beta'y \geq \tau$}\;
}
}}}
% \Return{WHAT\;}
\caption{\texttt{separation-G}\label{alg:separation-G}}
\end{algorithm}

 \fi

 \subsection{Removing redundant disjunctions} \label{sec:removal}
 
 The examples given in the proof of Theorem~\ref{th:nondominance} illustrate that removing redundant disjunctions could lead to faster separation and also to \DCs which dominate the \DCs obtained without the removal of such disjunctions:
 
 \begin{itemize}
     \item For $\hat y_2$, the sets $\mathcal P \cap D_i(\hat y_2)$ for $i \in\{2,3,4\}$ are empty. Thus we do not need to consider these disjunctions when defining \eqref{CG-SOCP}, which leads to a smaller SOCP in the separation procedure.
     \item For $\hat y_1$, the set $\mathcal P \cap D_3(\hat y_1)$ does not contain any bilevel-feasible solution, as it does not contain any integer solution. By removing the disjunction $D_3(\hat y_1)$,  a new  \DC $y \leq 2$ can be obtained by using cut-coefficient normalization with 1-norm. This new \DC dominates the \DC obtained when $D_3(\hat y_1)$ is included.
 \end{itemize}
 
 Thus, ideally, we would like to eliminate disjunctions $\mathcal{D}_i(\hat y)$ which do not contain any bilevel-feasible solution. Because this condition is very difficult to verify, as pointed out in~\cite[cf.\ Theorem 5 and Corollary 1]{fischetti2018use}, we could simply check whether 
 $\mathcal{D}_i(\hat y)$ is not satisfied for any point satisfying the variable bounds of $\mathcal{P}$ (thus, relaxing the condition of checking whether $\mathcal{D}_i(\hat y)$ is not satisfied for any point in $\mathcal{P}$) for $i = 1, \dots, m_2$. In particular, we could check whether
 \[  \sum_{j=1}^{n_1} \tilde x_j 
 > f_i - B^i \hat{y} \]
holds, where $\tilde x_j := \min\{ A^i_j x^+_j, A^i_j x^-_j\}$
 and 
% $\tilde y_\red{j} = \min\{ \red{B^i_j} y^+_i, \red{B^i_j} y^-_i\}$, 
where $x^+_j$%, $y^+_\red{j}$ 
%(respectively, 
and $x^-_j$%, $y^-_\red{j}$
%) 
are the upper and lower bounds imposed on $x$ 
%and $y$ 
inside of $\mathcal{P}$, respectively.
 %\todo[inline]{Here the previous notation was not fitting to our notation, so I have adapted it. @Ivana @Kübra, can you please double check whether this now really means what you want it to mean? (Because you two wrote this paragraph).} 
 However, this only considers each disjunction with variable bounds individually and may not be very effective. In what follows, we propose several other approaches, ordered by their computational effort.%\todo{make notation consistent: $\mathcal{D}_i(\hat y) $ vs $\mathcal{D}_i$} 

     \paragraph{Relaxation-based removal.} A disjunction $\mathcal{D}_i(\hat y)$ is redundant if $\mathcal{P} \cap \mathcal{D}_i(\hat y) = \emptyset$. Checking this requires solving a (small) SOCP. 
     %, which could also reduce to an LP in case $\mathcal{P}$ just contains linear constraints and $\mathcal{D}_i(\hat y)$ is linear. 
     If the disjunction can be removed, \eqref{CG-SOCP} is smaller. Moreover, each \DC which can be obtained when the disjunction is considered in \eqref{CG-SOCP}, can also be obtained when the disjunction is removed before solving \eqref{CG-SOCP}.
    % and the (potentially) resulting cut is not stronger than the cut obtained without removal of the disjunction. %\todo{Do we need to prove this? I: yes, because I think this is not true, because of degeneracy.}
    
     \paragraph{Integrality-based removal.} A disjunction $\mathcal{D}_i(\hat y)$ is redundant if $ \mathcal{P} \cap \mathcal{D}_i(\hat y) \cap \mathbb{Z}^n = \emptyset$. Checking this requires solving a (small) integer-SOCP. If such a  disjunction is removed, \eqref{CG-SOCP} is smaller and the (potentially) resulting \DC can dominate the \DC obtained without removal of the disjunction (see, e.g., the example discussed above for $\hat y_1$).
     
    \paragraph{Optimality-based removal.} A disjunction $\mathcal{D}_i(\hat y)$ is redundant if among the solutions in $\mathcal{P} \cap \mathcal{D}_i(\hat y) \cap \mathbb{Z}^n$, there is no solution that improves the current best objective-function value (say, $U\!B$) of the leader.
     \begin{theorem}\label{thm:redundant}
 Let $U\!B$ be the objective-function value of the best-known feasible solution for the bilevel program~\eqref{bilevel}. Let furthermore $\bar U_i$ be the optimal objective-function value of the  problem
 \[ \min \{ c'x + d' y : (x,y) \in \mathcal{P} \cap \mathcal{D}_i (\hat y) \cap \mathbb{Z}^n \}. 
 \]
 If $\bar U_i \ge U\!B$, then the disjunction $\mathcal{D}_i(\hat y)$ is redundant and can be discarded from $D(\hat{y},\mathcal{P})$ and thus also from  \eqref{CG-SOCP}. 
 \end{theorem}
 
 \begin{proof}
 If $\bar U_i \ge U\!B$, then any solution in $\mathcal{P} \cap \mathcal{D}_i(\hat y) \cap \mathbb{Z}^n$ is not better for the bilevel program~\eqref{bilevel} than the best-known feasible solution. Thus any better solution cannot be in $\mathcal{P} \cap \mathcal{D}_i(\hat y) \cap \mathbb{Z}^n$ and therefore no optimal solution is missed by removing $\mathcal{D}_i(\hat y)$.
 \ifArXiV %---------------------------------------------------------
\else %--------------------------------------------------------------
\qed
\fi %--------------------------------------------------------------
 \end{proof}
 
 If the disjunction can be removed, \eqref{CG-SOCP} is smaller and the (potentially) resulting \DC can  dominate the \DC obtained without removal of the disjunction. In particular, the removal can result in such a cut even if the integrality-based removal fails to remove a disjunction: For example, consider a slightly modified version of the instance in the proof of Theorem~\ref{th:nondominance}, where constraint \eqref{eq:gMB3} is replaced by $-4x+3y\geq -7$. In this case, for $\hat y_1=2$, we have $\mathcal{D}_3(\hat{y_1}): x \ge 7/2$ and the set $\mathcal P \cap D_3(\hat y_1)$ contains an integer solution, namely $(4,3)$. Thus, with the first two approaches of removing redundant disjunctions, we keep the disjunction $D_3(\hat y_1)$ and would get the \DC $-1.25x+3.1y\leq 5.7$ using cut-coefficient normalization with $1$-norm (see Figure~\ref{fig:rr1}). However, suppose that  we know, e.g., of the bilevel-feasible solution $(2,2)$ which has the leader objective-function value $U\!B=0$. Because $\bar U_3=1$, the condition of Theorem~\ref{thm:redundant} is satisfied, and the disjunction can be removed. The resulting \DC in this case is $y\leq 2$ using cut-coefficient normalization with $1$-norm (see Figure~\ref{fig:rr2}).

 \begin{figure}[tbh]
    \centering
    \begin{subfigure}{0.48\textwidth}
    \includegraphics[width=1.0\textwidth]{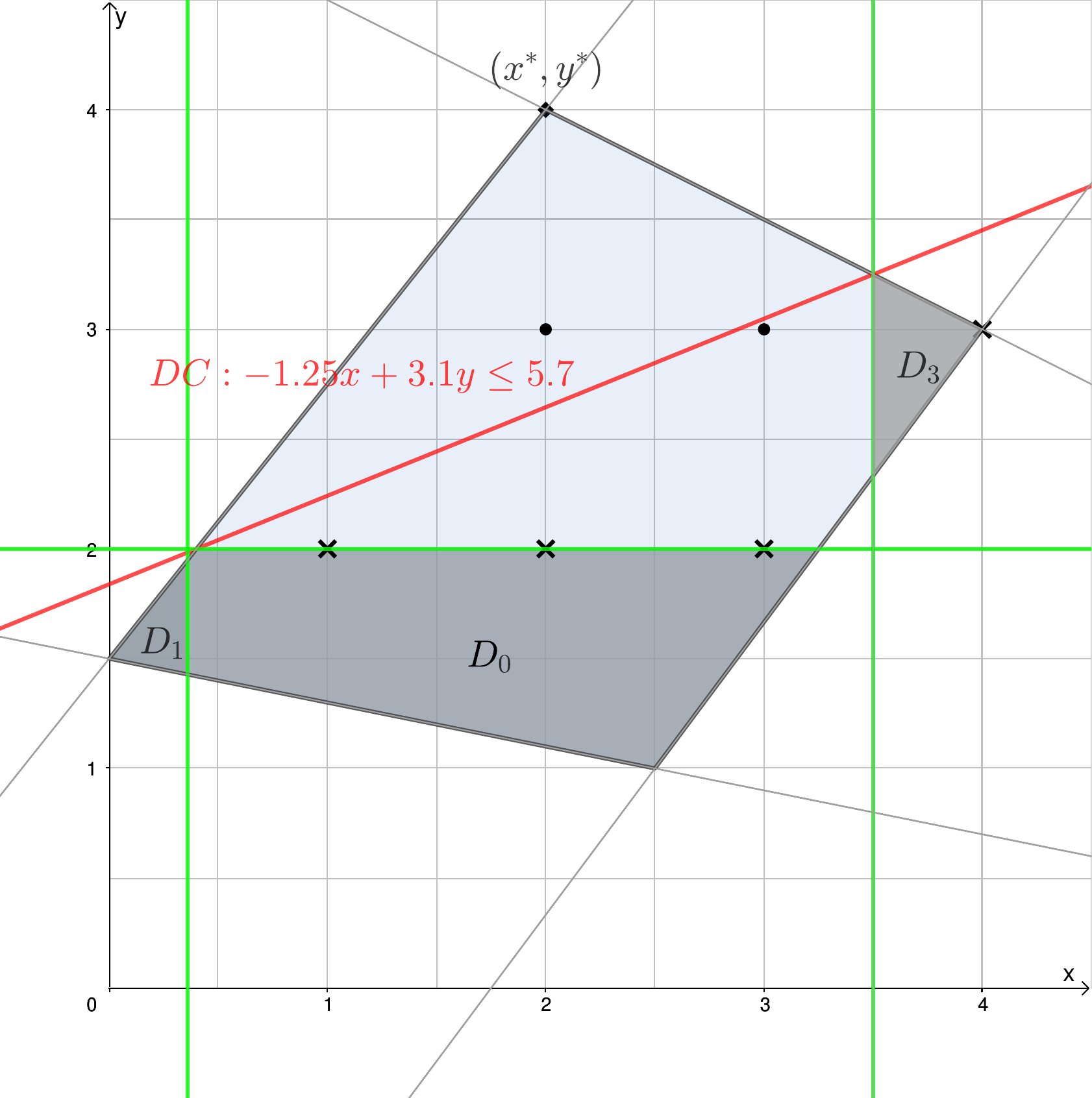}
    \caption{Example where both relaxation-based and integrality-based removal fail.\label{fig:rr1}}
    \end{subfigure}\quad
        \begin{subfigure}{0.48\textwidth}
     \includegraphics[width=1.0\textwidth]{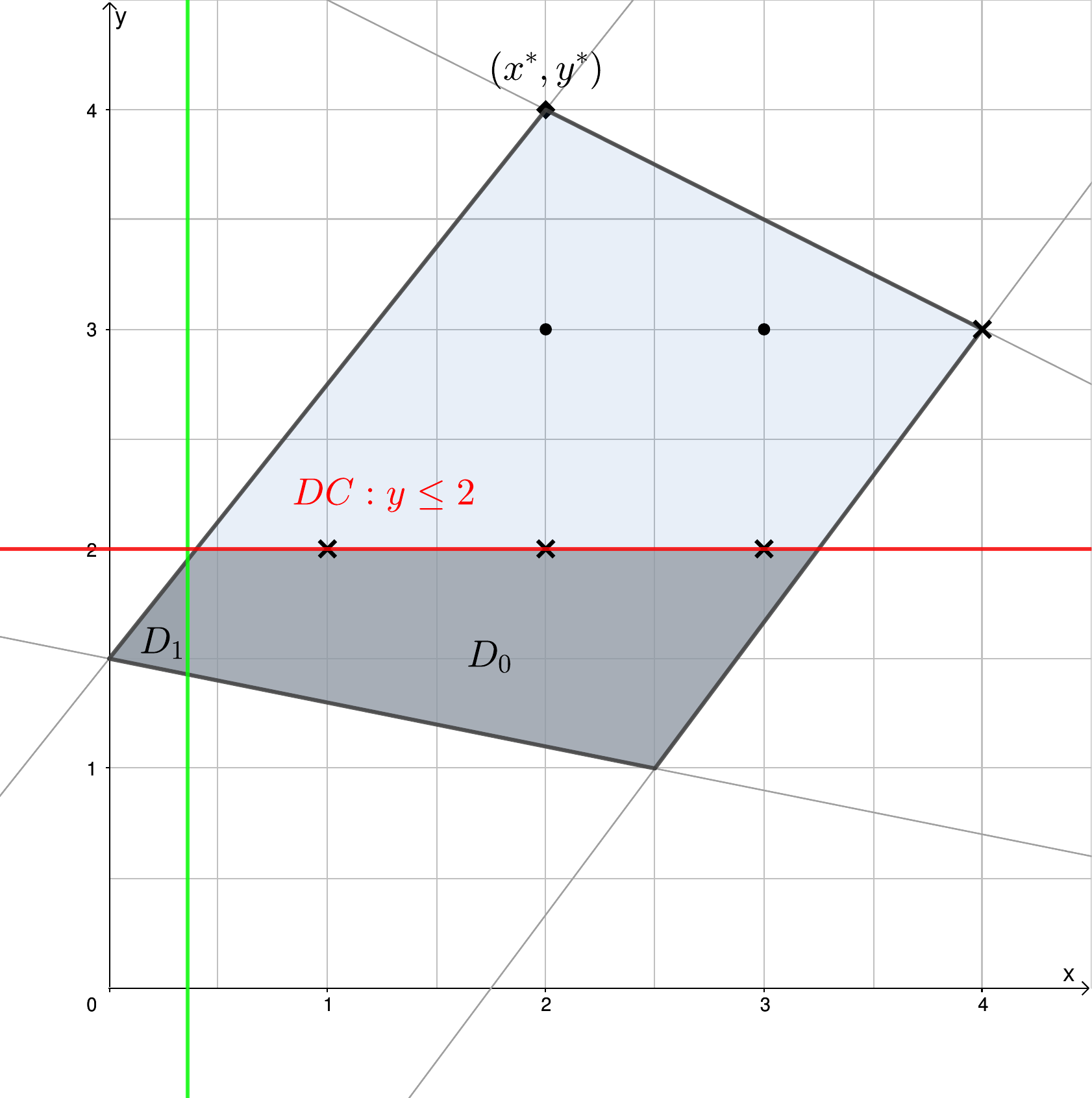}
     \caption{\DC derived after applying optimality-based removal to the disjunction $D_3(\hat y_1)$\label{fig:rr2}}
        \end{subfigure}
    \caption{An example illustrating optimality-based removal of disjunctions.}
    \label{fig:redundant}
\end{figure}

\subsection{Normalization}
\label{sec:normalization}

As mentioned before, we need to deal with the fact that the feasible region of \eqref{CG-SOCP}
is a cone. So \eqref{CG-SOCP} either has its optimum at the origin (implying that
%$(\hat{x},\hat{y})$
$(x^*,y^*)$ cannot be separated), or \eqref{CG-SOCP} is unbounded, implying that
there is a violated inequality, which of course we could scale by any positive
number  so as to make the violation as large as we like. The standard remedy for
this is to introduce a normalization constraint to \eqref{CG-SOCP}.
%, effectively 
%replacing
%each ray of the feasible region with an extreme point. The precise normalization
%used can have a very significant effect on the quality of the cut generated.
A typical good choice (see~\cite{Fischetti2011})
is to additionally impose 
%$\|(\bar{\pi}_1, \bar{\pi}_2, \tilde{\pi}_1, \tilde{\pi}_2, 
%\sigma, \rho  )\|_1\ \leq 1$,
\[ \|((\bar{\pi_i})_{i=0}^{m_2}, \; (\tilde{\pi}_i)_{i=0}^{m_2}, \;  \sigma, \; \rho  )\|_1\ \leq 1
\]
on \eqref{CG-SOCP}, 
i.e., that the 1-norm of the set of dual multipliers is unity. 
%%This can be implemented
%%with auxiliary variable vectors \red{$\tilde{\pi}_i^\uparrow$, $i=0,\dots,m_2$, }
%%$\tilde{\pi}_2^\uparrow$,
%$\rho^\uparrow$, and the linear constraints
%% $\tilde{\pi}_1 \leq \tilde{\pi}_1^\uparrow$,
%%  $-\tilde{\pi}_1 \leq \tilde{\pi}_1^\uparrow$,
%% $\tilde{\pi}_2 \leq \tilde{\pi}_2^\uparrow$,
%% $-\tilde{\pi}_2 \leq \tilde{\pi}_2^\uparrow$,
%% $\rho \leq \rho^\uparrow$,
%% $-\rho \leq \rho^\uparrow$,
%% $e' \bar{\pi_1} + e' \bar{\pi_2} + e' \tilde{\pi}_1^\uparrow +  e' 
%% \tilde{\pi}_2^\uparrow  -\sigma
%% + e' \rho^\uparrow = 1$. 
%\red{
%\begin{align*}
%\tilde{\pi}_i & \leq \tilde{\pi}_i^\uparrow, \quad i=0,\dots,m_2\\
%-\tilde{\pi}_i & \leq \tilde{\pi}_i^\uparrow, \quad i=0,\dots,m_2\\
%\rho & \leq \rho^\uparrow, \\
%-\rho & \leq \rho^\uparrow, \\
%\sum_{i=0}^{m_2}e' \bar{\pi_i} + \sum_{i=0}^{m_2} e' \tilde{\pi}_i^\uparrow %-\sigma + e' \rho^\uparrow & = 1. 
%\end{align*}
%}
%\todo{change notation, maybe shorten}
%but in our context, 
Because we are using a conic solver,
we can alternatively 
%\red{(and more easily and efficiently)}\footnote{to be seen, based on the experiments}  
impose
 %$\|(\bar{\pi}_1, \bar{\pi}_2, \tilde{\pi}_1, \tilde{\pi}_2, \sigma, \rho    )\|_2\ \leq 1$,
\[ \|((\bar{\pi_i})_{i=0}^{m_2}, \; (\tilde{\pi}_i)_{i=0}^{m_2}, \;  \sigma, \; \rho  )\|_2\ \leq 1,
\]
 which is just one constraint for a conic solver. 
 This kind of normalization, where the norm of the vector of all dual variables is bounded by 1, is called the \emph{standard} normalization (see, e.g., Lodi et al.~\cite{Lodi-et-al:2020}).% and can be implemented in a straight-forward way.
 
 %Thus, we will from now on consider normalization as part of \eqref{CG-SOCP}.
In~\cite{Lodi-et-al:2020}, not only the standard normalization but also \emph{uniform} normalization is concluded to be among the best normalizations in terms of numerical robustness. In uniform normalization, only the norm of the vector of dual variables corresponding to the constraints shared by all disjunctions (and not to the ones defining the disjunctions) is bounded, i.e., 
\[ \|((\bar{\pi_i})_{i=0}^{m_2}, \; (\tilde{\pi}_i)_{i=0}^{m_2} )\|_p\ \leq 1,
\]
where $p \in \{1,2\}$. Note, that~\cite{Lodi-et-al:2020} considered only a generalization of the $1$-norm, and not the $2$-norm, for both the standard and the uniform normalization.

Another alternative is \emph{cut-coefficient} normalization, where the norm of the cut-coefficients $\alpha$ and $\beta$ is bounded by one; so
\[ \|(\alpha, \; \beta )\|_p\ \leq 1
\]
is imposed, typically for some $p \in \{1,2\}$.
This may seem to be the most intuitive kind of normalization, as solving the \eqref{CG-SOCP} yields the desired cut-coefficients.

\paragraph{Theoretical considerations concerning normalization.}
To investigate the influence of normalization, we next present the duals of \eqref{CG-SOCP}. 
Without normalization, the dual has objective function zero and the feasible region is 
\begin{subequations}
\label{dual}
\begin{align}
&\textstyle \sum_{i=0}^{m_2}x_i = x^*  \label{deq1}\\
&\textstyle \sum_{i=0}^{m_2}y_i = y^* \label{deq2}\\
&\textstyle \sum_{i=0}^{m_2}\lambda_i = 1 \label{deq3}\\
&\lambda_i \ge 0 
&& \forall i =0,\ldots, m_2\label{deq4} \\
&\bar{M}x_i + \bar{N}y_i \ge \lambda_i\bar{h} 
&& \forall i =0,\ldots, m_2\label{deq5} \\
&\tilde{M}x_i + \tilde{N}y_i - \lambda_i \tilde{h} \in \mathcal{K^{**}}
&& \forall i =0,\ldots, m_2\label{deq6} \\
&\tilde{D}y_0 - \lambda_0\tilde{c} \in \mathcal{Q}\label{deq7}\\
&A^i x_i \leq \lambda_i (f_i - B^i \hat{y}  - 1) 
&& \forall i =1,\ldots, m_2,\label{deq8} 
\end{align}
\end{subequations}
i.e., the dual tries to find points $(x_i,y_i)$, such that $(x^*,y^*)$ is the sum of these points, and such that either $\lambda_i = 0$ or $\lambda_i > 0$ and $(\frac{1}{\lambda_i}x_i,\frac{1}{\lambda_i}y_i) \in \mathcal{P} \cap \mathcal{D}_i(\hat{y})$. As a consequence, \eqref{dual} is feasible if and only if $(x^*,y^*)$ is in 
$ \conv \left( \bigcup_{i=0}^{m_2} (\mathcal{P} \cap \mathcal{D}_i(\hat{y})) \right)$. Note that this corresponds exactly to the case that the primal  \eqref{CG-SOCP} does not find a violated cut, i.e., its optimal objective-function value  is zero. 

When deriving the duals of \eqref{CG-SOCP} with normalization, we assume that the normalization was imposed with the $p$-norm for $p \in \{1,2\}$. Let the $p^*$-norm be the dual norm of the $p$-norm, i.e. $p^* = 2$ for $p=2$ and $p^* = \infty$ for $p=1$. Note that normalizing \eqref{CG-SOCP} using the $p$-norm leads to a $p^*$-norm in the objective function of the dual.

In case of standard normalization, the dual of \eqref{CG-SOCP} is
\begin{subequations}
\label{dualS}
\begin{align}
- &\min \| ((\bar{\mu}_i)_{i=0}^{m_2}, \;  (\tilde{\mu}_i)_{i=0}^{m_2}, \; \mu_\sigma, \; \mu_\rho)\|_{p^*}\\
\st~ &\eqref{deq1}-\eqref{deq4} \\
&\bar{M}x_i + \bar{N}y_i \ge \lambda_i\bar{h} + \bar{\mu}_i
&& \forall i =0,\ldots, m_2\label{dSeq5} \\
&\tilde{M}x_i + \tilde{N}y_i - \lambda_i \tilde{h} - \tilde{\mu}_i \in \mathcal{K^{**}}
&& \forall i =0,\ldots, m_2\label{dSeq6} \\
&\tilde{D}y_0 - \lambda_0\tilde{c} - \mu_\rho \in \mathcal{Q}\label{dSeq7}\\
&A^i x_i \leq \lambda_i (f_i - B^i \hat{y}  - 1) - {\mu_\sigma}_i
&& \forall i =1,\ldots, m_2.\label{dSeq8} 
\end{align}
\end{subequations}

\begin{Observation}\label{ob:norm1}
The problem \eqref{dualS} is always feasible, and there is a feasible interior point due to the free variables 
$(\bar{\mu}_i)_{i=0}^{m_2}$, $(\tilde{\mu}_i)_{i=0}^{m_2}$,  $\mu_\sigma$ and $\mu_\rho$, which relax the constraints to be in $\mathcal{P}$ and to be in $\mathcal{D}_i(\hat{y})$. 
The optimal objective-function value of \eqref{dualS} is zero if and only if $(x^*,y^*)$ is in $ \conv \left( \bigcup_{i=0}^{m_2} (\mathcal{P} \cap \mathcal{D}_i(\hat{y})) \right)$, i.e., if and only if there is no violated \DC for $(x^*,y^*)$.
\end{Observation}

Note that Observation~\ref{ob:norm1} is compatible with what is observed in~\cite{Lodi-et-al:2020} when deriving split-cuts for mixed-integer \SOCP using disjunctive programming with \SOCP.
%In particular, and as already observed in~\cite{Lodi-et-al:2020}, \eqref{dualS} is always feasible and there is a feasible interior point due to the free variables 
%$(\bar{\mu}_i)_{i=0}^{m_2}$, $(\tilde{\mu}_i)_{i=0}^{m_2}$,  $\mu_\sigma$ and $\mu_\rho$, which relax the constraints to be in $\mathcal{P}$ and to be in $\mathcal{D}_i(\hat{y})$. 
%The optimal objective function of \eqref{dualS} is zero if and only if $(x^*,y^*)$ is in $ \conv \left( \bigcup_{i=0}^{m_2} (\mathcal{P} \cap \mathcal{D}_i(\hat{y})) \right)$, i.e., if and only if there is no violated \DC for $(x^*,y^*)$.

For the uniform normalization, the dual of \eqref{CG-SOCP} is
\begin{subequations}
\label{dualU}
\begin{align}
- &\min \| ((\bar{\mu}_i)_{i=0}^{m_2}, \;  (\tilde{\mu}_i)_{i=0}^{m_2})\|_{p^*}\\
\st~ & \eqref{deq1}-\eqref{deq4}, \eqref{dSeq5}, \eqref{dSeq6}, \eqref{deq7}, \eqref{deq8}.
\end{align}
\end{subequations}

\begin{Observation}\label{ob:norm2}
The problem \eqref{dualU} is not necessarily feasible, because only the constraints to be in $\mathcal{P}$ are relaxed with the variables 
$(\bar{\mu}_i)_{i=0}^{m_2}$ and $(\tilde{\mu}_i)_{i=0}^{m_2}$.
To be more precise, 
\eqref{dualU} is feasible if and only if $(x^*,y^*)$ is in $ \conv \left( \bigcup_{i=0}^{m_2} \mathcal{D}_i(\hat{y}) \right)$.
Due to the structure of our disjunctions (i.e., they are based on the follower constraints and the follower objective function), the point $(x^*,y^*)$ may not be in $ \conv \left( \bigcup_{i=0}^{m_2} \mathcal{D}_i(\hat{y}) \right)$ and thus \eqref{dualU} could be infeasible.

Furthermore, as for standard normalization, the optimal objective-function value of \eqref{dualU} is zero if and only if $(x^*,y^*)$ is in $ \conv \left( \bigcup_{i=0}^{m_2} (\mathcal{P} \cap \mathcal{D}_i(\hat{y})) \right)$.
\end{Observation}

We note that Observation~\ref{ob:norm2} is different compared to what the authors of~\cite{Lodi-et-al:2020} obtain in their setting, as the convex hull of the disjunction for split cuts is $\mathbb R^n$ and thus their resulting problem \eqref{dualU} is always feasible.

%Clearly \eqref{dualU} is not necessarily feasible, because only the constraints to be in $\mathcal{P}$ are relaxed with the variables $(\bar{\mu}_i)_{i=0}^{m_2}$ and $(\tilde{\mu}_i)_{i=0}^{m_2}$.
%To be more precise, 
%\eqref{dualU} is feasible if and only if $(x^*,y^*)$ is in $ \conv \left( \bigcup_{i=0}^{m_2} \mathcal{D}_i(\hat{y}) \right)$.
%This implies that for split cuts the dual for the uniform normalization is always feasible, see~\cite{Lodi-et-al:2020}.
%However, this is no longer the case for our disjunctions, as they are based on the follower constraints and thus their convex hull is not necessarily $\mathbb{R}^n$.
%If \eqref{dualU} is infeasible, the uniform normalization was not enough to bound \eqref{CG-SOCP}.
%As for standard normalization, the optimal objective-function value of \eqref{dualU} is zero if and only if $(x^*,y^*)$ is in $ \conv \left( \bigcup_{i=0}^{m_2} (\mathcal{P} \cap \mathcal{D}_i(\hat{y})) \right)$.

For the cut-coefficient normalization, the dual of \eqref{CG-SOCP}
is
\begin{subequations}
\label{dualC}
\begin{align}
- &\min \| (\mu_x, \; \mu_y)\|_{p^*}\\
\st~ &\textstyle \sum_{i=0}^{m_2}x_i = x^* + \mu_x\label{d3eq1}\\
&\textstyle \sum_{i=0}^{m_2}y_i = y^* + \mu_y \label{d3eq2}\\
& \eqref{deq3}-\eqref{deq8}, 
\end{align}
\end{subequations}
so geometrically \eqref{dualC} determines a point in $ \conv  \left( \bigcup_{i=0}^{m_2} (\mathcal{P} \cap \mathcal{D}_i(\hat{y})) \right)$ that minimizes the distance (in $p^*$-norm) to $(x^*,y^*)$. %, i.e., a projection of $(x^*,y^*)$. 

\begin{Observation}\label{ob:normC}
Problem \eqref{dualC} is feasible if and only if 
$\conv \left( \bigcup_{i=0}^{m_2} (\mathcal{P} \cap \mathcal{D}_i(\hat{y})) \right)$ is non-empty. If \eqref{dualC} is infeasible,  then all disjunctions are empty (i.e., redundant as described in Section~\ref{sec:removal}). 
\end{Observation}

%As a result, when using cut-coefficient normalization either \eqref{CG-SOCP} produces a disjunctive cut in case it is feasible, or we can deduce that all disjunctions are empty.

As a result of the investigation of the normalization, we know that Assumption~\ref{as:CGSOCPsolvable} is always satisfied with standard normalization. Unfortunately this is not the case with uniform and cut-coefficient normalization. We describe in Section~\ref{sec:implementationDetails} how we deal with this.

 \section{Solution methods using disjunctive cuts \label{sec:solution}}

We now present two solution methods based on \DCs: one applicable for the general bilevel program~\eqref{bilevel}, and one dedicated to a binary version of~\eqref{bilevel}.

\subsection{A branch-and-cut algorithm  %for \eqref{bilevel} 
\label{sec:bc}}

We propose to use the \DCs in a \BC algorithm to solve the bilevel program~\eqref{bilevel}. The \BC can be obtained by modifying any given continuous-relaxation-based \BB algorithm to solve the \HPR (assuming that there is an off-the-shelf solver for \ccrHPR that always returns an extreme optimal solution $(x^*,y^*)$ like e.g., a simplex-based \BB for a linear \ccrHPR 
\footnote{This assumption is without loss of generality, as we can outer approximate second-order conic constraints of $\mathcal{P}$ and get an extreme optimal point by a simplex method.}).

In particular, we adapt the \BB algorithm in the following way:
Use \ccrHPR as initial relaxation $\mathcal P$ at the root-node of the \BC. Whenever a solution $(x^*,y^*)$ which is integer is encountered in a \BC node, call the \DC separation. If a violated \DC is found, add the \DC  to the set $\mathcal P$
(which also contains, e.g., variable fixing by previous branching decisions, previously added globally or locally valid \DCs, \ldots)
of the current \BC node, otherwise the solution is feasible and the incumbent can be updated. 
Note that \DCs are only locally valid except the ones from the root node, because $\mathcal{P}$ includes branching decisions.
%Finiteness of this algorithm follows using similar arguments as in the proof of Theorem~\ref{th:finite}, a proof is omitted due to space reasons. 
If $\mathcal{P}$ is empty
or optimizing over $\mathcal{P}$ leads to an objective-function value that is greater than the objective-function value of the current incumbent, we fathom the current node.
%\blue{Also if a solution $(x^*,y^*)$ is encountered in a \BC node such that its objective-function value is larger than the objective-function value of the current incumbent, we fathom the node.}
In our implementation, we also use \DC separation for fractional $(x^*,y^*)$ as described in Section~\ref{sec:chooseyhat} for strengthening the relaxation. %In addition, for any heuristic solution we attempt to generate a violated DC. If this is not possible for a bilevel-infeasible solution, we discard the solution. 
%\todo{Ivana: 
%shall we mention that DCs are only locally valid (once we fix $x$ and $y$ variables due to branching, DCs are no longer globally valid, do you agree? 
%Also, we should mention that if the optimal value on $\mathcal{P}$ is bigger than the current incumbent, we fathom the node (standard \BB stuff, but needed for completeness). } 
%\todo{mention reject incumbent in the results section}
%\todo{infeasibility issue: if optimizing over $\mathcal{P}$ becomes infeasible we fathoming the node}
%\todo{Add reference for \BB?}

\begin{theorem}\label{th:finite}
%Algorithm~\ref{alg:cutting} 
The \BC solves the bilevel program~\eqref{bilevel} in a finite number of 
\BC-iterations under our assumptions.
\end{theorem}

\begin{proof}

%First, we show that the \BC works correctly if the problem ~\eqref{bilevel} is infeasible. There are two potential sources of infeasibility: Either the 

First, suppose that the \BC terminates, but the integer solution $(x^*,y^*)$ is not bilevel feasible. This is not possible, as by Theorem~\ref{tm:getCut} and the observations thereafter, the \DC generation procedure finds a violated cut to cut off the integer point $(x^*,y^*)$ in this case. 

Next, suppose that the \BC terminates and the solution $(x^*,y^*)$ is bilevel feasible, but not optimal. This is not possible, because by construction, the \DCs never cut off any bilevel-feasible solution (or in case of optimality-based removal, any bilevel-feasible solution, which has a better leader objective-function value than the currently best-known solution) of the current subtree.

Finally, suppose that the \BC never terminates. This is not possible, as all variables are integer and bounded due to Assumption~\ref{as:bound1}, thus there is only a finite number of nodes in the \BC tree. Moreover, this means that there is also a finite number of integer points $(x^*,y^*)$, thus we solve the follower problem and $\eqref{CG-SOCP}$ a finite number of times.
%In each of these nodes we call the \DC separation at most once, therefore  
%\todo{finite number}
The follower problem is discrete and can therefore be solved in a finite number of iterations. 
\ifArXiV %---------------------------------------------------------
\else %--------------------------------------------------------------
\qed
\fi %--------------------------------------------------------------
%Moreover, due to our assumptions, the relaxation problem at each node and also the follower problem in  can be solved in a finite number of iterations. 
\end{proof}

\subsection{A cutting-plane algorithm for binary \IBNPs}
\label{sec:intcut}

The \DCs can be directly used in a cutting-plane algorithm under the following assumption. 
\begin{Assumption}\label{as:binary}
All variables in the bilevel program~\eqref{bilevel} are binary variables.
\end{Assumption}
The algorithm is detailed in Algorithm~\ref{alg:cutting}. 
It starts with the \HPR as initial relaxation of \VFR, which is solved to optimality. Then the chosen \DC separation routine (either \texttt{O} or \texttt{G}) is called to check if the obtained integer-optimal solution is feasible for constraint \eqref{eq:inequValueFunction}. If not, the obtained \DC is added to the relaxation to cut off the optimal solution, and the procedure is repeated with the updated relaxation. 

Due to Assumption~\ref{as:binary}, each obtained binary optimal solution is an extreme point of the convex hull of \ccrHPR, and thus due to Theorem~\ref{tm:getCut}, a violated cut will be produced by the \DC separation if the solution is not bilevel feasible.
Note that without Assumption~\ref{as:binary}, i.e., if variables are allowed to be integer and not just binary, an optimal solution may not be an extreme point of \ccrHPR. In this case, Theorem~\ref{tm:getCut} does not apply. Thus, we cannot guarantee that the \DC separation finds a violated cut. As a consequence, our proposed cutting-plane algorithm only works for binary instances.

\SetKwRepeat{Do}{do}{while}
\begin{algorithm}[tbh]
\LinesNumbered
\SetKwInOut{Input}{Input}\SetKwInOut{Output}{Output}
\Input{An instance of problem \eqref{bilevel} where all variables are binary} 
\Output{An optimal solution $(x^*,y^*)$} 
%\BlankLine 
{
$\mathcal R \gets$ \HPR; $\mathcal P \gets$ set of feasible solutions of \ccrHPR; $\texttt{violated} \gets True$\;
\Do{\texttt{violated}}{
$\texttt{violated} \gets False$\;
solve $\mathcal R$ to optimality, let $(x^*,y^*)$ be the obtained optimal solution\;
call \texttt{separation} for $(x^*,y^*)$ and $\mathcal{P}$ with strategy \texttt{O} or \texttt{G}\; 
\If{a violated cut is found for $(x^*,y^*)$}
{
$\texttt{violated} \gets True$; add the violated cut to $\mathcal R$ and to $\mathcal P$\;
}
}
\Return{$(x^*,y^*)$}
}
% \Return{WHAT\;}
\caption{\texttt{cutting-plane}\label{alg:cutting}}
\end{algorithm} 

\section{Computational analysis}
\label{sec:comp}

In this section, we present computational results to empirically compare methods and strategies proposed in Sections~\ref{sec:comp_methodology} and~\ref{sec:solution}. We also assess computational difficulties in the presence of multiple linking constraints, or in the 
presence of integer instead of binary variables. Finally, we compare our new DC-based branch-and-cut with the state-of-the-art \MIBLOP-solver \MIX~\cite{fischetti2017new,fischetti2018use}.

\subsection{Instances}
In our computations, we  consider two sets of instances: the quadratic bilevel covering problem (QBCov) instances (originally studied in~\cite{Gaar-et-al:2022} and extended here with multiple linking constraints) and a new additional set of quadratic bilevel multiple knapsack problem (QBMKP) instances derived from the SAC-94 library~\cite{sac94}. The instances are available at \url{https://msinnl.github.io/pages/instancescodes.html}.
All instances can be described as 
\begin{subequations}
\label{qcbp}
\begin{align}
&\min  ~ \exSymb{c}' x + \exSymb{d}'y  \\
\st &~\exSymb{M} x + \exSymb{N} y  \ge \exSymb{h} \label{qcbp:pureLeaderCon} \\
&y  \in \arg \min \{ y' \exSymb{R} y : \exSymb{A} x + \exSymb{B} y  \ge \exSymb{f}, ~y \in \{0,1\}^{n_2} 
\}\\
&x \in \{0,1\}^{n_1},\label{eq:BQKP} 
\end{align}
\end{subequations}
where 
$\exSymb{c} \in \R^{           n_1}$, 
$\exSymb{d} \in \R^{           n_2}$, 
$\exSymb{M} \in \R^{m_1 \times n_1}$, 
$\exSymb{N} \in \R^{m_1 \times n_2}$, 
$\exSymb{h} \in \R^{           m_1}$, 
$\exSymb{R} =\exSymb{V}'\exSymb{V} \in \Z^{n_2 \times n_2}$, 
%$\exSymb{a} \in \Z^{n_1}$, 
%$\exSymb{b} \in \Z^{n_2}$, and 
%$\exSymb{f} \in \Z$.
$\exSymb{A} \in \Z^{ m_2 \times  n_1}$, 
$\exSymb{B} \in \Z^{m_2 \times n_2}$, and 
$\exSymb{f} \in \Z^{m_2}$.

\paragraph{The QBCov instances.} 
In this setting, we chose $m_2=1$, in which case the problem can be seen as the covering-version of the quadratic bilevel knapsack problem studied by Zenarosa et al.\ in~\cite{zenarosa2021exact}. Indeed, \cite{zenarosa2021exact} considers a single leader variable ($n_1=1$) and no coupling constraints at the leader ($m_1=0$), and with a quadratic non-convex leader objective function. 
%\eqref{qcbp:pureLeaderCon}). 
The linear variant of such a bilevel knapsack-problem is studied in, e.g., \cite{BrotcorneHM09,Brotcorne}. We note that~\cite{BrotcorneHM09,Brotcorne,zenarosa2021exact} propose problem-specific solution approaches.

We generated  40 random instances in the following way. We considered $n_1 = n_2$ for $n_1+n_2 = n\in \{20, 30, 40, 50\}$, and we study instances with no (as in~\cite{zenarosa2021exact}) and with one leader constraint \eqref{qcbp:pureLeaderCon}, so $m_1 \in \{0,1\}$. For each $n$, we created five random instances for each $m_1 \in \{0,1\}$.
Furthermore, we chose all entries of $\exSymb{c}$, $\exSymb{d}$, $\exSymb{M}$, $\exSymb{N}$, $\exSymb{A}$, and $\exSymb{B}$ uniformly at random from $\{0,1,\dots,99\}$. The values of $\exSymb{h}$ and $\exSymb{f}$ (which are scalars for these instances) were set to the sum of the entries of the corresponding rows in the constraint matrices divided by four.
The matrix $\exSymb{V}\in \R^{n_2 \times n_2}$ has integer entries chosen uniformly at random from the set $\{0,1,\dots,9\}$. We extended this data set from~\cite{Gaar-et-al:2022} with 40 new instances generated in the same way, by choosing $m_2=2$.

\paragraph{The QBMKP instances.}
These instances were derived from the multiple knapsack problem (MKP) instances from SAC-94 library~\cite{sac94} which is a benchmark library containing 0/1 MKP instances. From there we chose 50 instances and
%~\cite{freville1990hard} (instances \texttt{hp*} and \texttt{pb*}), \cite{petersen1967computational} (instances \texttt{pet*}), 
% %\cite{senju1968approach} (instances \texttt{sento*})\todo{K:we don't use sento},
%~\cite{shih1979branch} (instances \texttt{weish*}), and~\cite{weingartner1967methods} (instances \texttt{weing*}). 
generated 300 new instances of the QBMKP as follows.
%\todo{Kuebra: please double check}

The instances have 2 to 10 constraints and 10 to 105 items. For each instance of this data set, we first constructed two different QBMKP instances by keeping all but the last $m_2$ constraints at the leader problem, where $m_2 \in \{1,2\}$. The first 50 or 75\% of items are associated to leader variables $x$, the remaining ones are associated to follower variables $y$. The coefficients of the original MKP objective function are assigned to the leader.
%\todo{with a min direction}. 
Each budget constraint of the starting MKP, say $a'x + b'y \le f$, 
%\todo{is $\exSymb{f}$ okay here, we already have $c$} 
is translated into a covering constraint of type \eqref{qcbp:pureLeaderCon} as 
$a'x + b'y \ge e'a + e'b - \exSymb{f}$ (where $e$ is the vector of all ones). To generate the positive semidefinite matrix $\hat R$ determining the follower objective function, we follow a procedure proposed in~\cite{kleinert2021outer}.
We first randomly generated quadratic matrices $V$ of suitable size whose entries are chosen uniformly at random from $\{-\sigma, \dots, \sigma\}$  where $\sigma :=  \lceil \sqrt[4]{||\exSymb{d}||_{\infty}} \rceil$ and then set $\hat R:= V'V$. This allows to keep the order of magnitude for the coefficients of the objective function of the follower similar to those of the leader. Following this procedure, we obtained 200 binary instances, 100 with one linking constraint and 100 with two linking constraints.

Lastly, we generated integer instances where $m_2=1$, all decision variables take value in $\{0,\ldots, 5\}$, and the right-hand side of the previously generated covering constraints are multiplied by two. Together with the two choices of variable assignments to the leader and the follower problems, we obtained 100 integer QBMKP instances. 
%\todo{Kubra: first linearization, then integer instance description}

%\todo{how many binary with single linking, binary with multiple linking, and integer?}

\paragraph{Linearization of instances.}

The structure of~\eqref{qcbp} allows for an easy linearization of the convex nonlinear terms in the binary instances % follower objective function
using a standard McCormick linearization to transform the starting problem into an \MIBLOP. This allows us to compare the performance of our algorithm against a state-of-the-art \MIBLOP-solver \MIX from Fischetti et al.~\cite{fischetti2017new,fischetti2018use}.  

%\red{to get a first impression of whether our development of a dedicated solution approach for \IBNPs exploiting nonlinear techniques is a promising endeavour.}\footnote{Move to conclusions}

\subsection{Computational environment}

All experiments were executed on a single thread of an Intel Xeon E5-2670v2 machine with 2.5 GHz processor 
with a memory limit of 10 GB and a time limit of 600 seconds.
Our \BC algorithm and our cutting-plane algorithm both are implemented in C++. They make use of 
IBM ILOG CPLEX 12.10 (in its default settings, except for disabling presolve so that we can access the original \HPR formulation, setting the MIP gap tolerance to zero, and  running it single-threaded) as branch-and-cut framework in our \BC algorithm and as solver for $\mathcal R$  in our cutting-plane algorithm. 
During the \BC, CPLEX's internal heuristics are allowed and a bilevel-infeasible heuristic solution is just discarded if a violated cut cannot be obtained. For calculating the follower solution $\hat y$ for a given $x^*$, we also use CPLEX. For solving \eqref{CG-SOCP}, we use MOSEK~\cite{mosek} 9.2 in its default settings, except for running it always  single-threaded and using always the primal solver to avoid numerical issues. The solver \MIX against which we compare was run with CPLEX 12.9, which is the newest CPLEX version compatible with this solver.

\subsection{Implementation details}
\label{sec:implementationDetails}

\paragraph{Update of \texorpdfstring{$\mathcal{P}$}{P}.}
For both the B\&C and the cutting-plane algorithm, we start with \ccrHPR as initial $\mathcal{P}$ and do not update it with dynamically added \DCs. This is in line with the recent implementation of \DCs with split-cuts for mixed-integer \SOCP in~\cite{Lodi-et-al:2020}, and prevents potential numerical instabilities. 
However, in the \BC we update $\mathcal{P}$ with the local %node-bounds of 
variable bounds at the current node in the \BC-tree. Thus the obtained \DCs are only locally valid (i.e., in the current subtree) and are added as \emph{locally valid cuts}. Doing this is numerically safe, as variable bound constraints are already present in the original problem and they just need to be updated, i.e., the number of constraints remains the same.

We note that technically, in Theorem~\ref{tm:getCut}, we need that the current $\mathcal{P}$ incorporates all added \DCs and the local %node-bounds
variable bounds to ensure that $(x^*,y^*)$ is an extreme point, and thus can be separated in case it is bilevel infeasible. However, in our computational experiments, we never encountered any issues when not including previous \DCs in $\mathcal{P}$. On the other hand, updating $\mathcal P$ with the local bounds was crucial to make the separation work for integer instances. 

\paragraph{Solving the follower problem to obtain $\hat{y}$.}
During the  separation of both integer and fractional points $(x^*,y^*)$, while solving the follower problem, we make use of the follower objective-function value $q(y^*)$, by setting it as an upper cutoff value. This is a valid approach because a violated \DC exists only if $\Phi(x^*)<q(y^*)$. 

Furthermore, we use the solution-pool feature of CPLEX. This means that CPLEX keeps feasible solutions obtained in previous iterations and tries to use them as initial solution. % in case they are also feasible for the current follower problem. 
As a consequence, when using the separation option \texttt{G}, CPLEX may not need to start the solution process, as a solution from the solution-pool might be feasible for the current follower problem.

\paragraph{Checking for redundant disjunctions.}
{
%\todo{Elli: is it really such a good idea to emphasise that we only have an LP (and not SOCP constraints in the leader)? If we would have that, the check would not be an LP anymore.}

In our implementation, we only check the redundancy of linear disjunctions, i.e., $\mathcal{D}_i(\hat y)$ for $i=1,\ldots, m_2$, because
%Firstly, for the problems we choose, the problem to be solved to check (relaxation-based) redundancy is an LP when checking the linear disjunctions, and a QP when checking the objective disjunction $\mathcal{D}_0(\hat y)$. Secondly, 
we observed in the preliminary tests that the objective disjunction is almost never redundant as $(x^*,\hat{y})$ usually gives a feasible solution to $\mathcal{P} \cap \mathcal{D}_i(\hat y)$.

When solving the problems to detect redundancy, we keep the original leader objective function and define the feasible region as $\mathcal{P} \cap \mathcal{D}_i(\hat y)$ or $\mathcal{P} \cap \mathcal{D}_i(\hat y)\cap \mathbb{Z}^n$ together with the local variable bounds which will be used to obtain a \DC. 
%We set a solution limit of one, as we only need a proof of infeasibility for removing a disjunction. 
For optimality-based removal, we solve the problem with an upper cutoff value $z^*-10^{-5}$, where $z^*$ is the current leader incumbent objective value. Note that redundancy of a disjunction is indicated by infeasibility of the corresponding detection problem. Thus as soon as CPLEX finds a feasible solution for a detection problem, we know that the disjunction is not redundant. Hence, we set the CPLEX parameter \texttt{solution limit} to one.

}

\paragraph{Solving \eqref{CG-SOCP}.}
%Moreover, all DCs generated in the B\&C fashion are added as \emph{locally valid cuts}. Therefore, before calling the \eqref{CG-SOCP}, we include variable local bounds into the description of $\mathcal{P}$. 
To avoid numerical issues, whenever a coefficient of a DC is close enough to zero (i.e.,  absolute value less than $5\cdot10^{-6}$), we round it  to zero and adapt the right-hand side of the DC to maintain a valid cut. 

Unless mentioned differently, we use standard normalization with the $2$-norm, where Assumption~\ref{as:CGSOCPsolvable} is satisfied. 
{Whenever the dual of \eqref{CG-SOCP} is infeasible in the case of uniform normalization, we take $\alpha$, $\beta$ and $\tau$ from an unbounded ray of the primal (which is provided by MOSEK) as \DC. To prevent numerical issues, we scale any unbounded ray in such a way that $\| (\alpha, \beta)\|_2 = 1$ holds.
If the dual of \eqref{CG-SOCP} with cut-coefficient normalization is infeasible, all disjunctions are empty as described in Section~\ref{sec:normalization}. Thus we add the always violated cut $\alpha = 0$, $\beta=0$ and $\tau = 1$ in this case.}

\paragraph{Separation.}

%\todo{explain redundant implementation, branching?}
%        \item turn rounding on
%        \item MOSEK always primal
%        \item nodebounds always 1, because they are needed for the integer anyway, and they do not make so much difference for binary instances
%        \item \textcolor{green}{branching preference on x variables}

We %determine 
set the minimum acceptable cut violation $\varepsilon$ described in Section~\ref{sec:separation} to $10^{-6}$ for our experiments.
We control the number of cuts added for separating fractional points as follows. At the root node of the B\&C tree, we add as many cuts as needed, i.e., we check if we are able to cut off the current point with a \DC until no violated cut can be obtained. At all other nodes, we add at most one \DC and then proceed to branching, as the separation procedure could be time consuming. %\todo{please provide numbers}

Finally, in our \BC implementation, we also have to deal with integer solutions that are produced by the internal heuristics of CPLEX. In this case, we do not necessarily have a useful $\mathcal P$ for separation at hand. Thus, if the produced heuristic solution is bilevel infeasible and we fail to cut it off with a \DC, we just use the \texttt{reject}-feature of CPLEX to reject this solution (this prevents CPLEX from updating the incumbent with the heuristic solution).

\subsection{Numerical results}

%\todo{Kuebra, Markus: charts (tables maybe appendix)}
We start by assessing the performance of the B\&C approach, and  by evaluating how the choice of separation strategy, removal of redundant disjunctions, and normalization affect the overall performance.  
We then compare the B\&C against two alternatives: the cutting-plane method described in Section~\ref{sec:solution} and the state-of-the-art \MIBLOP solver \MIX from~\cite{fischetti2017new,fischetti2018use}. 
We conducted these experiments on {140} instances from our benchmark set that contain only binary variables, and a single linking constraint. Finally, we extended the benchmark set, and we also demonstrate the performance of our B\&C when applied to instances with multiple linking constraints, and with integer variables.

\paragraph{{Performance of different ingredients of the B\&C algorithm.}}
We discussed different separation procedures  in Section~\ref{sec:chooseyhat}. 
While executing our B\&C algorithm, we consider four different settings for the separation of cuts:
\begin{itemize}
    \item \texttt{IO}: only integer solutions are separated using  strategy \texttt{O},
    %for the \texttt{separation} routine, 
  %  this is the \emph{base setting} in which only integer   points $(x^*,y^*)$ are separated using an optimal follower solution, $\hat y \in \Omega(x^*)$ (strategy \texttt{O} for the \texttt{separation} routine), %\todo{dots even though no full sentences?}
    \item \texttt{IFO}: both integer and fractional solutions are separated using strategy \texttt{O},
    %for the \texttt{separation} routine, 
    \item \texttt{IG}:  only integer solutions  are separated using  strategy \texttt{G},
    %for the \texttt{separation} routine, 
    %only integer solutions are separated using the {first feasible solution $\hat y \in F(x^*)$ found such that $q(\hat y) < q(y^*)$ when solving the follower problem (strategy \texttt{G}} for the \texttt{separation} routine),
    \item \texttt{IFG}: both integer and fractional solutions are separated using  strategy \texttt{G}.
    %for the \texttt{separation} routine. 
\end{itemize} 

\begin{figure}[tbp]
\caption{ECDFs reporting runtimes and final gaps for four different separation strategies of the B\&C, over binary instances with one linking constraint.
\label{fig:separation}
%Final (up to legend) Plots: Experiment 1, choosing Sep option, for all binary, one disjunction instances (=140 instances started, 110 remaining after removing unknown status instances)
}
\begin{center}

\includegraphics[width=0.5\textwidth]{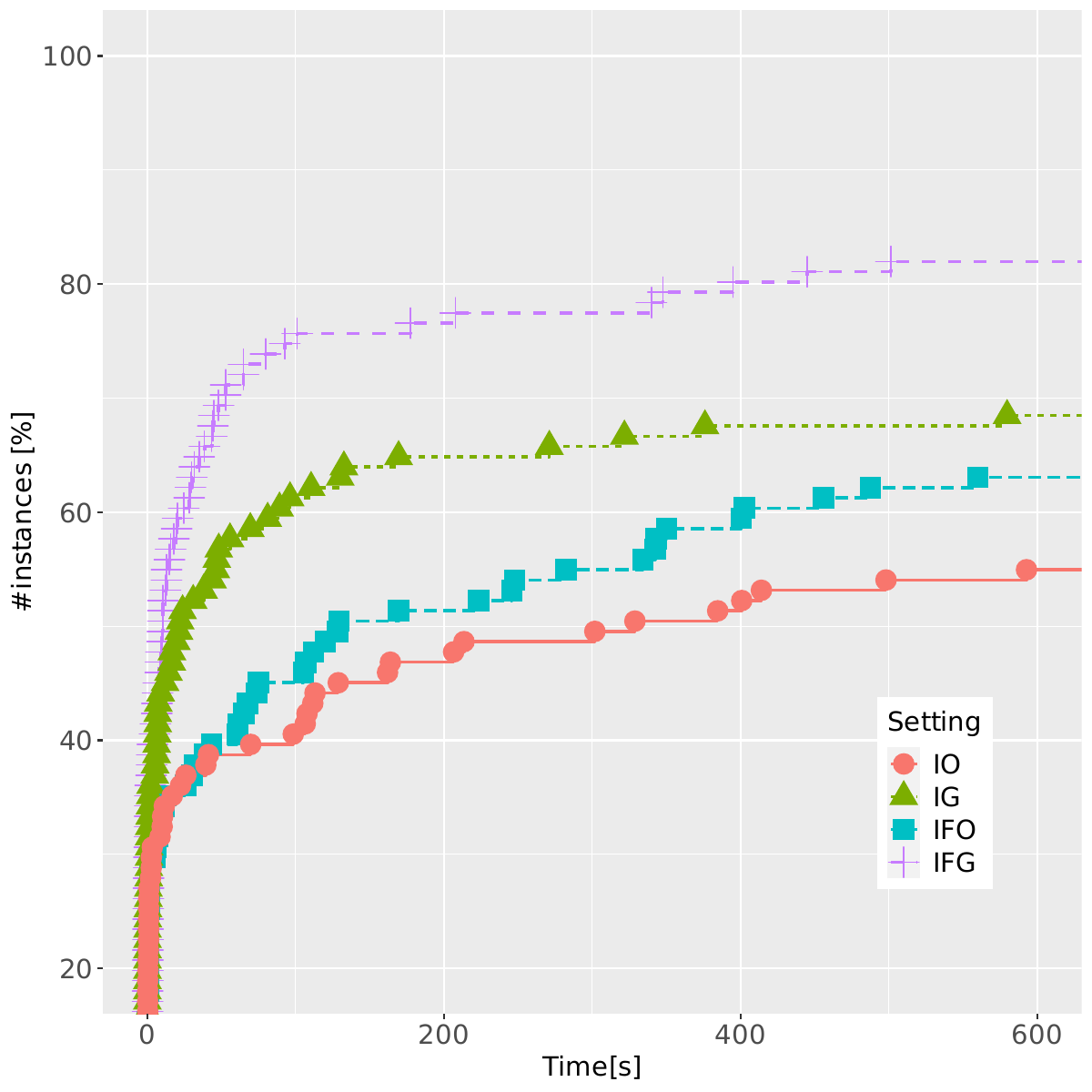}%
\includegraphics[width=0.5\textwidth]{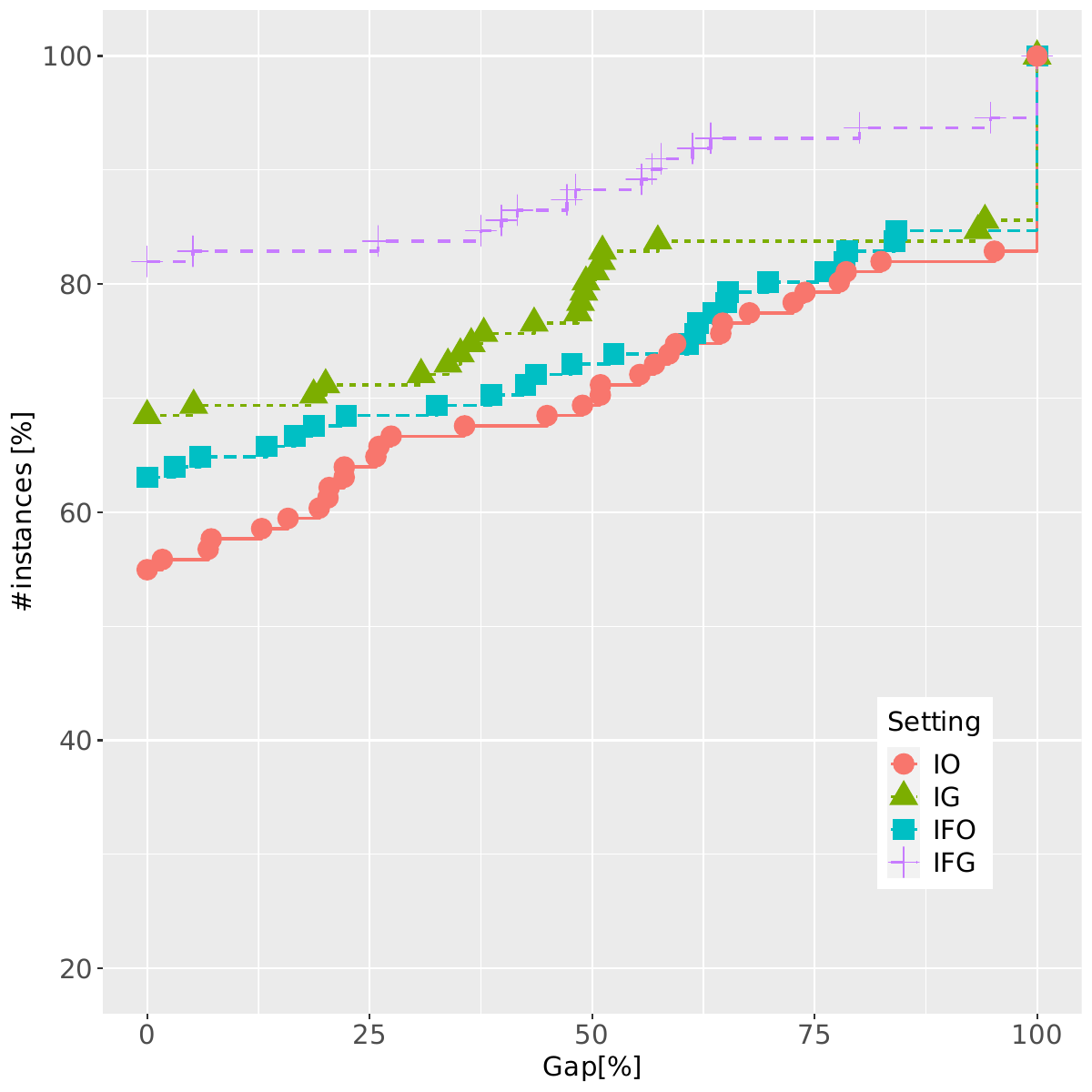}%
\end{center}
\end{figure}

In Figure~\ref{fig:separation} we compare these four settings for the B\&C. We show the empirical cumulative distribution functions (ECDFs) w.r.t. the runtimes and final gaps.
The gaps are defined in the following way.
If an instance is proven to be infeasible, we define the gap to be zero.
Otherwise, if no feasible solution is found, we define the gap to be 100.
If a feasible solution is found, then 
the gap is calculated as $100 (z^*-LB)/z^*$, where $z^*$ and $LB$ are the best-known objective-function value of a feasible solution and the lower bound, respectively. 
Note that by the construction of our instances, the value zero is a trivial lower bound for all feasible instances. Thus the gaps will always range between zero and 100.
The ECDFs with e.g., runtimes can be interpreted as the percentage of instances (shown in $y$-axis) that can be solved within a certain amount of time (depicted in the $x$-axis). In order to have a fair comparison, out of 140 instances from this benchmark set, we only consider 111 instances for which at least one of the methods was either able to find a feasible solution, or to prove  infeasibility (for 29 instances from QBMKP, the feasibility status remains unknown). We observe that the best-performing setting is 
%the one in which both integer and fractional points are separated using the greedy separation routine (setting \texttt{IFG}). 
\texttt{IFG}.
This can be explained by the fact that non-optimal follower solution may also provide a strong DCs (cf.\ Theorem~\ref{th:nondominance}) and by the significant savings in separation time (as we are avoiding to solve the follower problem to optimality). {Our preliminary results reported in~\cite{Gaar-et-al:2022} did not identify  \texttt{IFG} as the best setting, because there was no control mechanism implemented to limit the number of separated DCs for each fractional point, which may cause overloading of the master problem.
This is now regulated as described in Section~\ref{sec:implementationDetails}.
%This is now regulated by adding only a small number of violated cuts at fractional points. 
}

\begin{figure}[tbp]
\caption{
%Final (up to legend) Plots: Experiment 2, choosing discard redundant, for all binary, one disjunction instances (=140 instances) 
ECDFs reporting runtimes and final gaps for four strategies for the removal of redundant disjunctions, over binary instances with one linking constraint. \label{fig:removal}}
\begin{center}

\includegraphics[width=0.5\textwidth]{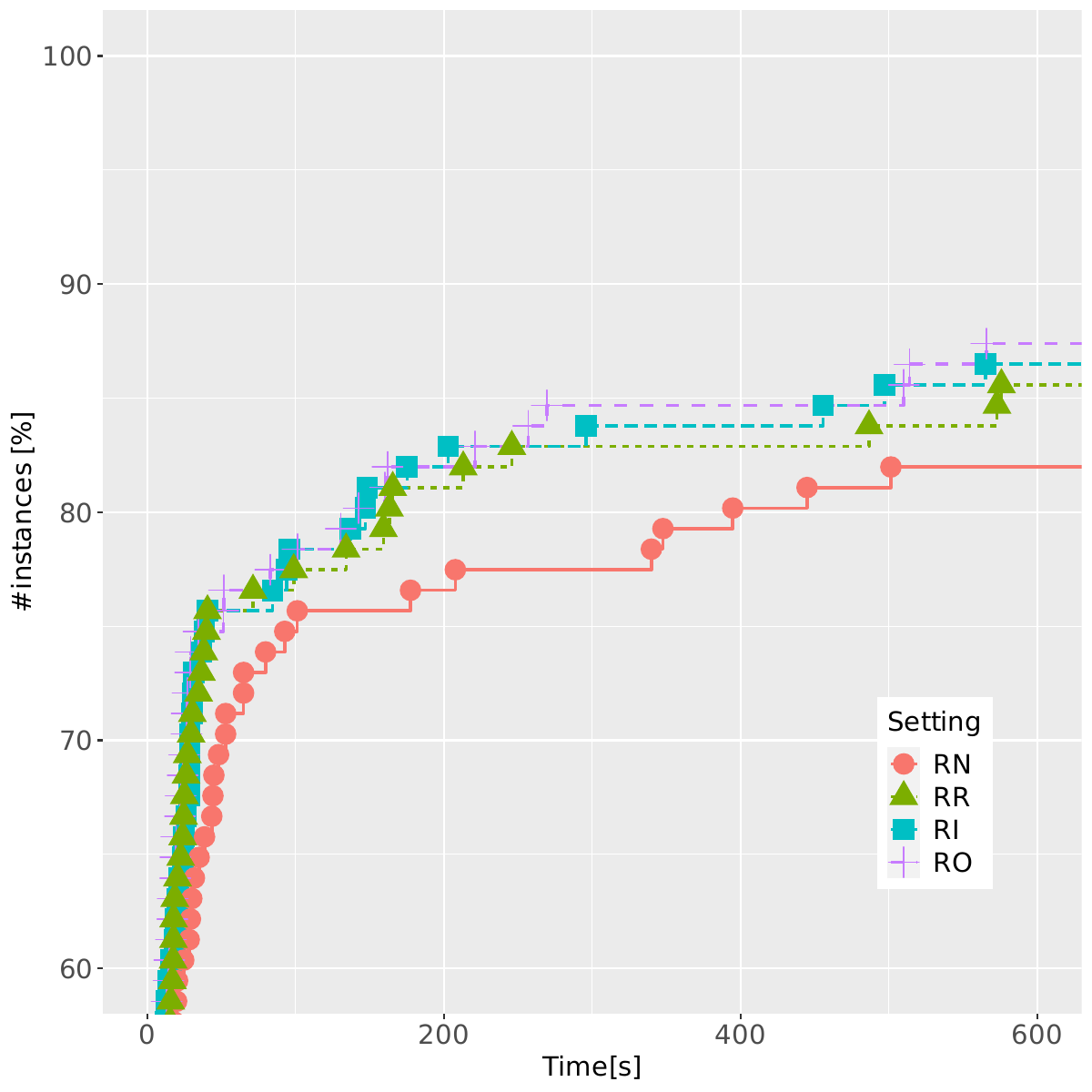}%
\includegraphics[width=0.5\textwidth]{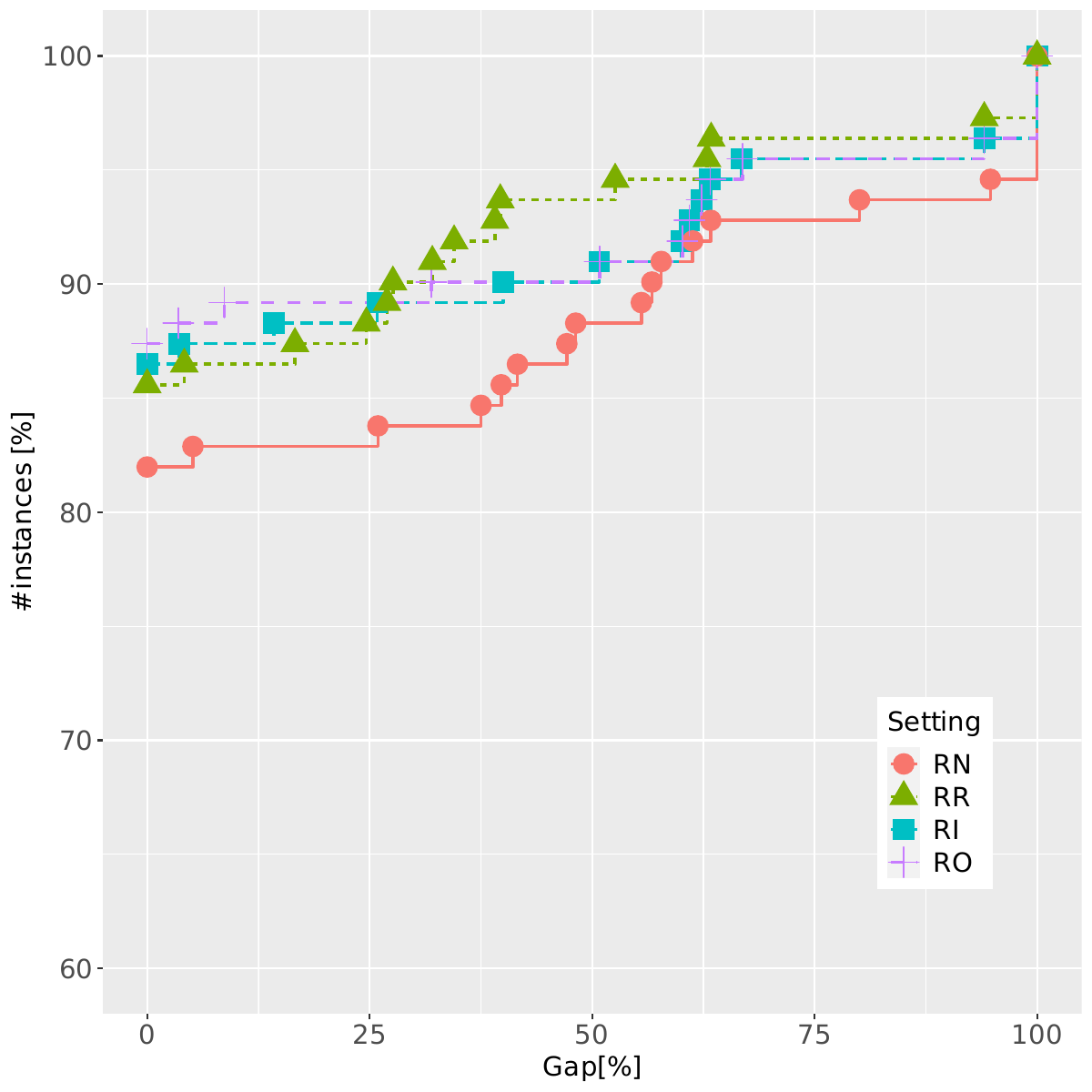}%
\end{center}
\end{figure}

In what follows, we continue with the setting \texttt{IFG}, and investigate how the potential removal of redundant disjunctions (discussed in Section~\ref{sec:removal}) affects the overall performance. Two additional ECDFs are reported in Figure~\ref{fig:removal} for the settings denoted as 
\texttt{RN} (no removal of redundant disjunctions), \texttt{RR} (relaxation-based removal), \texttt{RI} (integrality-based removal), and \texttt{RO} (optimality-based removal). We notice that the runtime can be improved, even when only the simplest strategy \texttt{RR} is applied. Also, the other more computationally-expensive removal strategies do not deteriorate the runtime, and in particular, they help to significantly improve the final gaps. For example, when including the strongest strategy \texttt{RO}, for 90\% of the instances, the final gap remains below 10\%, whereas without the removal, the respective gap can be as large as 50\%. Even though for the few most difficult instances, the best final gaps are obtained when using the \texttt{RR} strategy, we decided to continue with the rest of experiments using \texttt{RO} as a more stable and robust setting. 

\begin{figure}[tbp]
\caption{%Final (up to legend) Plots: Experiment 3, choosing norm, for all binary, one disjunction instances (=140 instances); N1 = standard 1-norm, N2 = standard 2-norm, N3 = uniform 1-norm, N4 = uniform 2-norm, N5 = cut-coeff 1-norm, N6 = cut-coeff 2-norm
ECDFs reporting runtimes and final gaps for six different normalization strategies for~\eqref{CG-SOCP}, over binary instances with one linking constraint. \label{fig:norm}} 
\begin{center}
\includegraphics[width=0.5\textwidth]{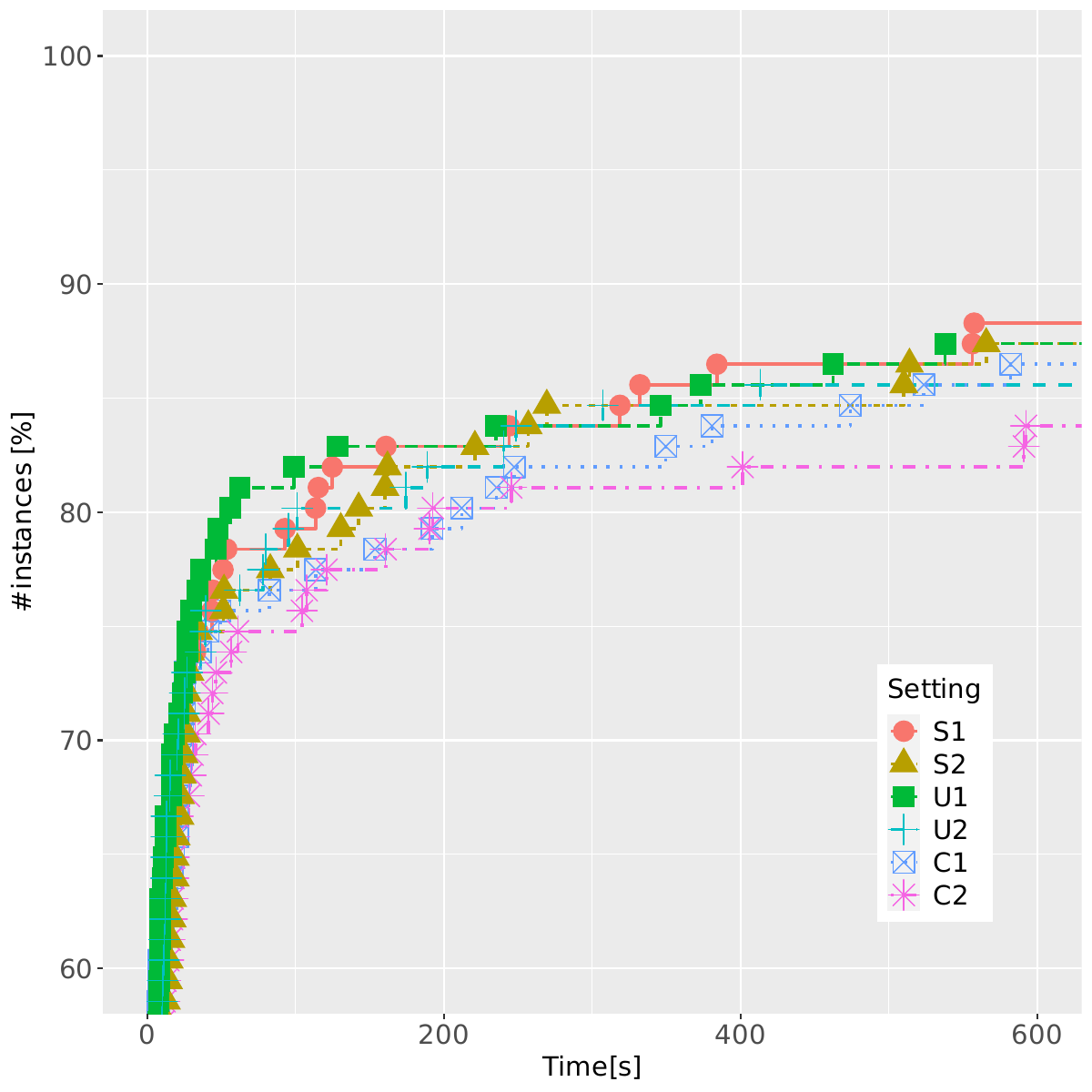}%
\includegraphics[width=0.5\textwidth]{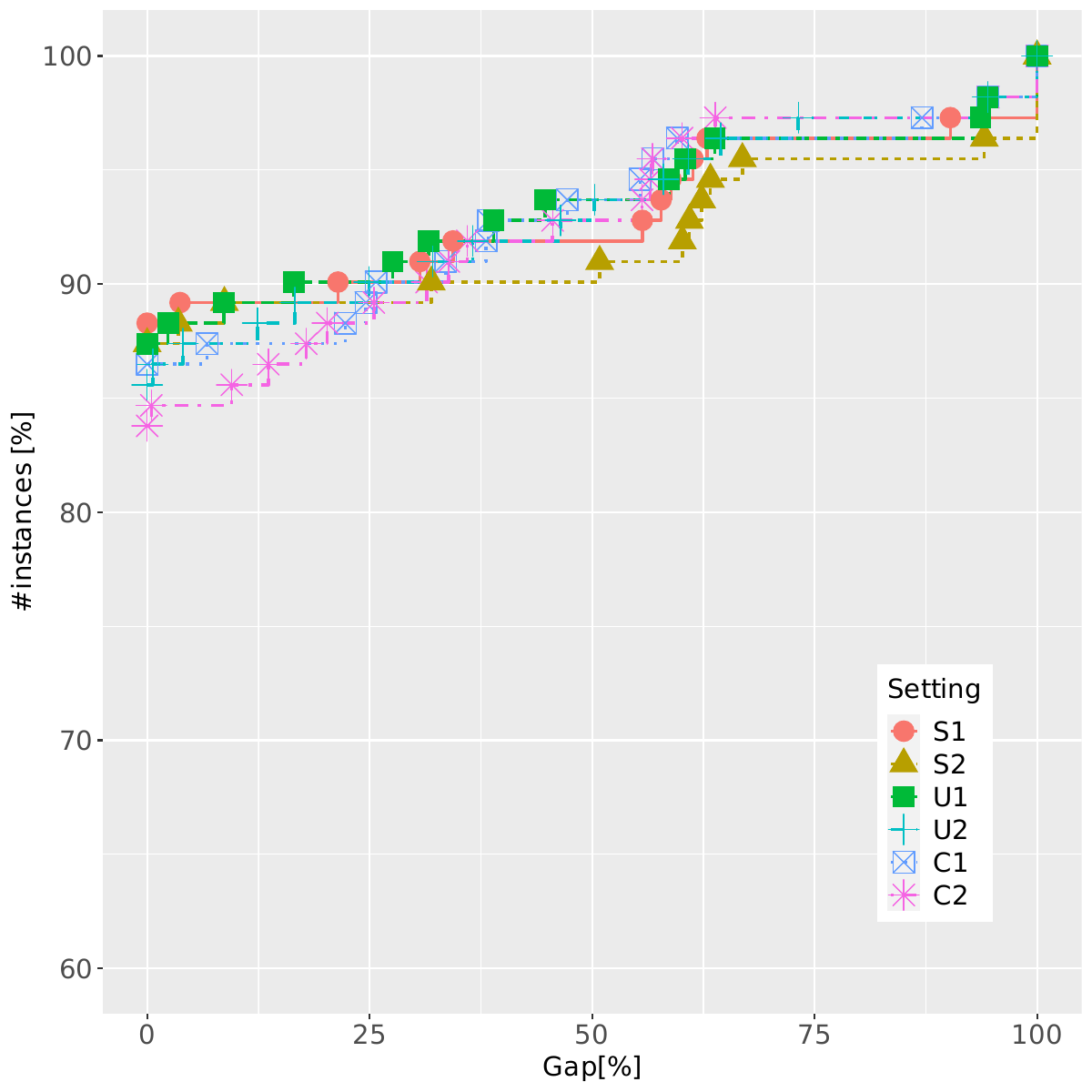}%

\end{center}
\end{figure}

Finally, when focusing on possible enhancement of the B\&C procedure, we look at the effectiveness of normalization strategies for~\eqref{CG-SOCP} presented in Section~\ref{sec:normalization}. The corresponding ECDFs are shown in Figure~\ref{fig:norm}. The letters \texttt{S}, \texttt{U} and \texttt{C} stand for standard, uniform and cut-coefficient normalization strategies, respectively, followed by $p \in \{1,2\}$ denoting the type of norm used. We observe that all normalization strategies perform very similarly, with two consistent trends: the worst-performing one is \texttt{C} (which is also in line with the known results from the literature on DCs, see, e.g.,~\cite{Lodi-et-al:2020}), and the 2-norm is always outperformed by the 1-norm. The latter can be explained by the sparsity and better numerical stability of cuts produced using the 1-norm. 
Thus, based on these experiments, we decide to use the strategy  \texttt{IFG} combined with \texttt{RO} and \texttt{S1} as our best %\todo{should we write "best" here?} 
setting (denoted as \texttt{BC-best} in the following). We denote by \texttt{BC-base} the setting where \texttt{IO} is combined with \texttt{RN} and \texttt{S2} 
%for~\eqref{CG-SOCP}
(the setting which was also used in our earlier study in~\cite{Gaar-et-al:2022}).

\begin{figure}[tbp]
\caption{ECDFs reporting runtimes and final gaps for the base and the best versions of the B\&C and cutting-plane methods, and the MIBLP solver, over  binary instances with one linking constraint.
\label{fig:method_comparison}}
\begin{center}

\includegraphics[width=0.5\textwidth]{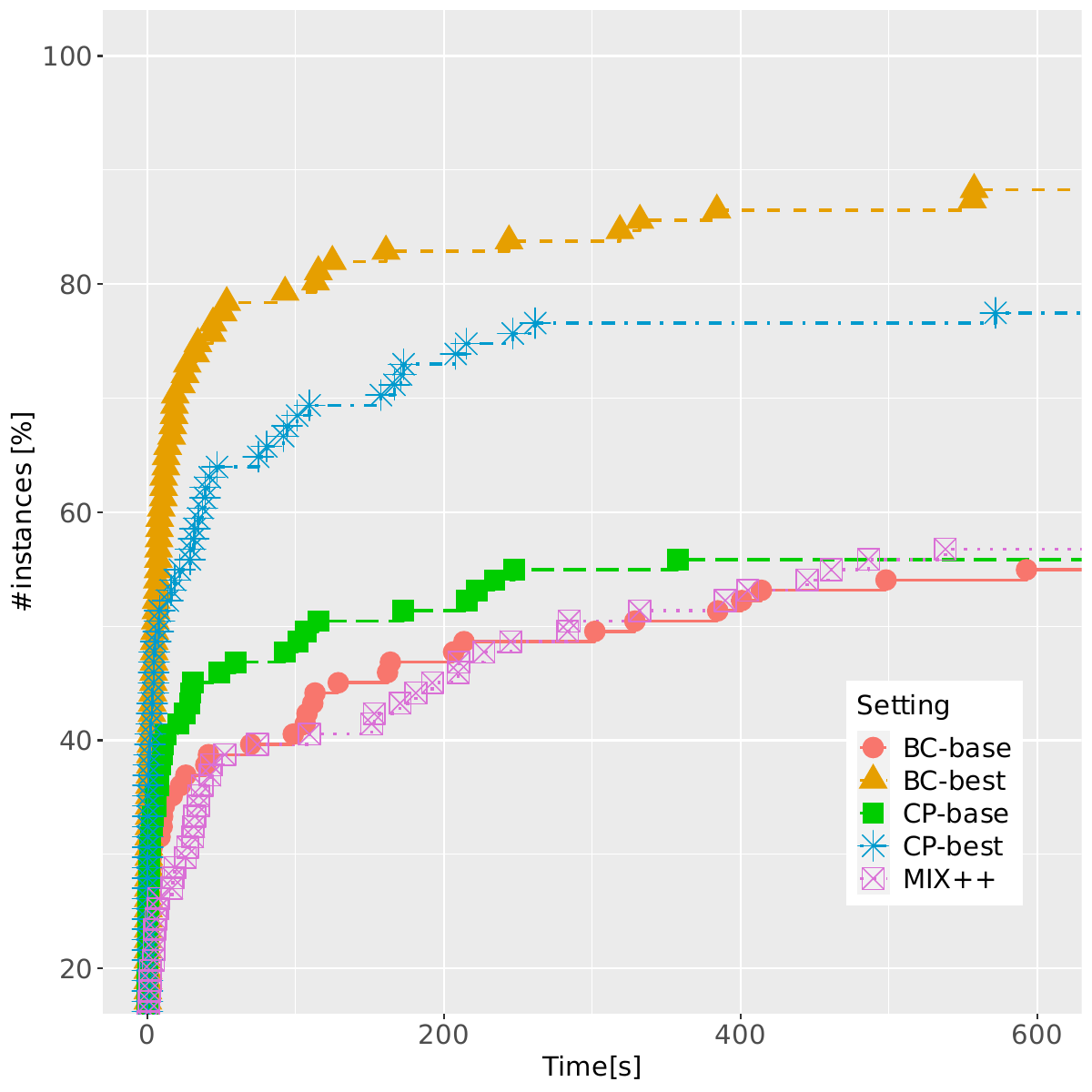}%
\includegraphics[width=0.5\textwidth]{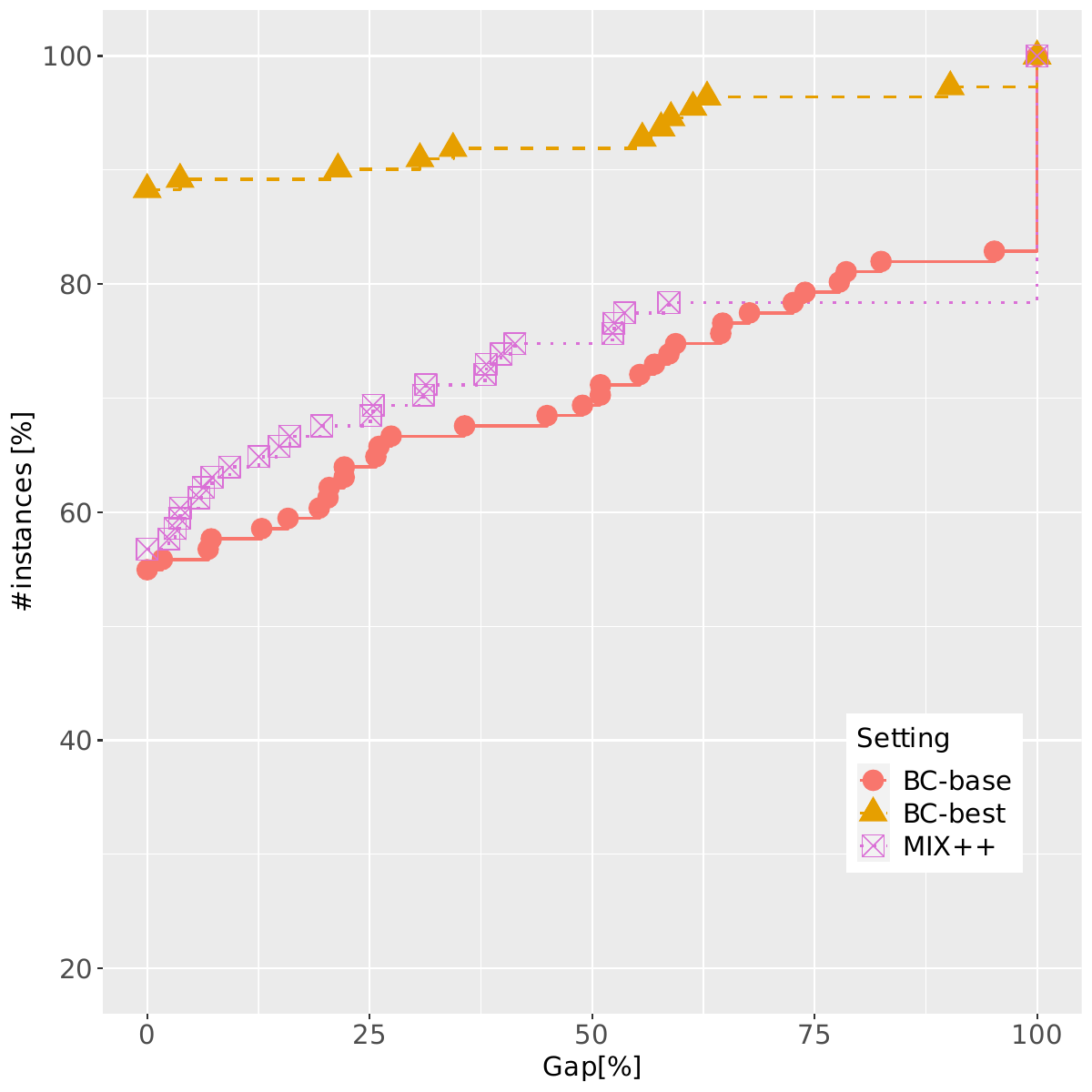}%
\end{center}
\end{figure}

\begin{table}[tbp]
\centering \setlength{\tabcolsep}{3pt}
\caption{Average results of five methods over 40 QBCov instances with one linking constraint.
\label{tab:method_comparison_cover}
}
\begingroup\scriptsize
\begin{tabular}{@{}rrrrrrrrrrrrrr@{}}
  \hline
  \multicolumn{1}{c}{$n$} & \multicolumn{1}{c}{\texttt{Setting}} & \multicolumn{1}{c}{$t$} & \multicolumn{1}{c}{\texttt{Gap}} & \multicolumn{1}{c}{\texttt{Gap*}} & \multicolumn{1}{c}{\texttt{RGap}}  & \multicolumn{1}{c}{\texttt{RGap*}} & \multicolumn{1}{c}{\texttt{nNode}} & \multicolumn{1}{c}{\texttt{nICut}} & \multicolumn{1}{c}{\texttt{nFCut}} & \multicolumn{1}{c}{\texttt{nRed}} & \multicolumn{1}{c}{$t_F$} & \multicolumn{1}{c}{$t_S$} & \multicolumn{1}{c}{\texttt{nSol}} \\ \hline

  \multirow{5}[2]{*}{20} & \texttt{BC-base} & 2.0 & 0.0 & 0.0 & 56.0 & 16.2 & 165.8 & 61.9 & - & - & 1.3 & 0.4 & 10/10 \\ 
   & \texttt{BC-best} & 1.1 & 0.0 & 0.0 & 66.6 & 16.4 & 93.0 & 17.1 & 46.1 & 40.3 & 0.4 & 0.3 & 10/10 \\ 
   & \texttt{CP-base} & 1.0 & - & 0.0 & - & - & - & 15.3 & - & - & 0.3 & 0.2 & 10/10 \\ 
   & \texttt{CP-best} & 0.2 & - & 0.0 & - & - & - & 10.6 & - & 1.0 & 0 & 0.1 & 10/10 \\ 
   & \texttt{MIX++} & 2.2 & 0.0 & 0.0 & 14.5 & 14.5 & 25.0 & - & - & - & - & - & 10/10 \\ 
   \hline \multirow{5}[2]{*}{30} & \texttt{BC-base} & 154.8 & 2.6 & 0.9 & 70.5 & 15.1 & 4461.1 & 1879.9 & - & - & 143.8 & 10 & 9/10 \\ 
   & \texttt{BC-best} & 3.6 & 0.0 & 0.0 & 53.4 & 14.0 & 507.8 & 32.8 & 285.9 & 246.8 & 1.5 & 1.7 & 10/10 \\ 
   & \texttt{CP-base} & 70.7 & - & 0.0 & - & - & - & 131.3 & - & - & 12.5 & 0.8 & 10/10 \\ 
   & \texttt{CP-best} & 16.7 & - & 0.0 & - & - & - & 77.4 & - & 1.1 & 0.3 & 0.4 & 10/10 \\ 
   & \texttt{MIX++} & 116.5 & 0.0 & 0.0 & 15.6 & 15.6 & 300.2 & - & - & - & - & - & 10/10 \\ 
  \hline \multirow{5}[2]{*}{40} & \texttt{BC-base} & 305.4 & 16.3 & 7.7 & 100.0 & 21.1 & 5717.0 & 2058.4 & - & - & 289.4 & 14.7 & 6/10 \\ 
   & \texttt{BC-best} & 25.0 & 0.0 & 0.0 & 90.0 & 18.9 & 3282.1 & 196.2 & 1759.9 & 1629.5 & 11.0 & 12.7 & 10/10 \\ 
   & \texttt{CP-base} & 256.2 & - & 2.3 & - & - & - & 249.2 & - & - & 52.6 & 1.9 & 6/10 \\ 
   & \texttt{CP-best} & 117.6 & - & 0.9 & - & - & - & 182.4 & - & 0.7 & 0.9 & 1.2 & 9/10 \\ 
   & \texttt{MIX++} & 456.9 & 20.2 & 8.7 & 32.2 & 21.8 & 368.2 & - & - & - & - & - & 4/10 \\ 
  \hline \multirow{5}[2]{*}{50} & \texttt{BC-base} & 551.2 & 46.9 & 27.8 & 100.0 & 36.2 & 7513.8 & 2618.2 & - & - & 526.4 & 23.2 & 2/10 \\ 
   & \texttt{BC-best} & 233.6 & 3.1 & 2.8 & 100.0 & 33.5 & 23479.3 & 1073.1 & 14859.2 & 9488 & 95.5 & 128.1 & 9/10 \\ 
   & \texttt{CP-base} & 549.3 & - & 17.9 & - & - & - & 418.8 & - & - & 153.8 & 3.9 & 1/10 \\ 
   & \texttt{CP-best} & 423.1 & - & 11.8 & - & - & - & 374.7 & - & 0.5 & 2.0 & 3.1 & 4/10 \\ 
   & \texttt{MIX++} & 600.0 & 53.1 & 31.6 & 57.8 & 38.1 & 245.5 & - & - & - & - & - & 0/10 \\ 
   \hline
\end{tabular}
\endgroup
\end{table}

\begin{table}[tbp]
\centering \setlength{\tabcolsep}{3pt}
\caption{Average results of five methods over %100 (
71  QBMKP instances with one linking constraint.
\label{tab:method_comparison_MKP}
}
\begingroup\scriptsize
\begin{tabular}{@{}rrrrrrrrrrrrrr@{}}
  \hline
  \multicolumn{1}{c}{$n_1/n$} & \multicolumn{1}{c}{\texttt{Setting}} & \multicolumn{1}{c}{$t$} & \multicolumn{1}{c}{\texttt{Gap}} & \multicolumn{1}{c}{\texttt{Gap*}} & \multicolumn{1}{c}{\texttt{RGap}}  & \multicolumn{1}{c}{\texttt{RGap*}} & \multicolumn{1}{c}{\texttt{nNode}} & \multicolumn{1}{c}{\texttt{nICut}} & \multicolumn{1}{c}{\texttt{nFCut}} & \multicolumn{1}{c}{\texttt{nRed}} & \multicolumn{1}{c}{$t_F$} & \multicolumn{1}{c}{$t_S$} & \multicolumn{1}{c}{\texttt{nSol}} \\ \hline
  \multirow{5}[2]{*}{0.50} & \texttt{BC-base} & 417.5 & 34.2 & 23.8 & 85.8 & 30.0 & 5847.7 & 2497.0 & - & - & 403.4 & 13.0 & 11/26 \\ 
   & \texttt{BC-best} & 71.2 & 3.6 & 3.6 & 78.6 & 13.6 & 5151.6 & 342.3 & 6415.3 & 4874.7 & 31.0 & 36.4 & 24/26 \\ 
   & \texttt{CP-base} & 357.2 & - & 20.7 & - & - & - & 257.4 & - & - & 131.2 & 5.3 & 12/26 \\ 
   & \texttt{CP-best} & 155.3 & - & 19.2 & - & - & - & 164.5 & - & 39.8 & 0.8 & 0.9 & 19/26 \\ 
   & \texttt{MIX++} & 341.2 & 20.4 & 14.3 & 35.9 & 22.4 & 199.5 & - & - & - & - & - & 16/26 \\ 
  \hline \multirow{5}[2]{*}{0.75} & \texttt{BC-base} & 302.8 & 38.2 & 18.1 & 66.3 & 23.7 & 17841.1 & 6581.2 & - & - & 258.9 & 39.6 & 23/45 \\ 
  & \texttt{BC-best} & 155.1 & 14.5 & 7.0 & 62.8 & 17.8 & 9657.0 & 600.8 & 9427.2 & 5600.4 & 63.2 & 84.1 & 35/45 \\ 
   & \texttt{CP-base} & 294.2 & - & 14.8 & - & - & - & 171.0 & - & - & 66.4 & 1.4 & 23/45 \\ 
   & \texttt{CP-best} & 191.5 & - & 10.7 & - & - & - & 154.2 & - & 8.7 & 0.8 & 1.3 & 31/45 \\ 
  & \texttt{MIX++} & 316.0 & 38.4 & 14.6 & 45.7 & 20.6 & 712.7 & - & - & - & - & - & 23/45 \\ 
   \hline
\end{tabular}
\endgroup
\end{table}

%IPCO figures
% \begin{figure}[tbp]
% \begin{center}
% \includegraphics[width=0.5\textwidth]{Runtime_new.png}%
% \includegraphics[width=0.5\textwidth]{Gap_new.png}
% \end{center}
% \caption{Runtimes and final optimality gaps for the quadratic bilevel covering problem under our different settings and the benchmark solver \MIX.}
% \label{figure:Runtime-Gap}
% \end{figure}

\paragraph{Comparison against alternative methods.}
For binary \IBNPs we proposed an alternative cutting-plane algorithm in Section~\ref{sec:intcut}. This algorithm can be implemented with both separation strategies (\texttt{G} and \texttt{O}). Moreover, the removal of redundant disjunctions and normalization strategies can be fine-tuned as well. 
%In the following, when referring to \texttt{BC-best} we refer to \texttt{IFG} combined with \texttt{RO} and \texttt{S1} for~\eqref{CG-SOCP}. When referring to \texttt{BC-base} we refer to \texttt{IO} combined with \texttt{RN} (no removal of disjunctions) and \texttt{S2} for~\eqref{CG-SOCP} (the setting which was also used in our earlier study in~\cite{Gaar-et-al:2022}).
In Figure~\ref{fig:method_comparison}, we compare the B\&C results with the two settings (\texttt{BC-best} versus \texttt{BC-base}) against the results obtained by the cutting-plane algorithm (\texttt{CP-best} which involves \texttt{G}, \texttt{RO}, and \texttt{S1} versus \texttt{CP-base} which involves \texttt{O}, \texttt{RN}, and \texttt{S2}) as well as a state-of-the-art \MIBLOP solver \MIX of Fischetti et al.~\cite{fischetti2017new,fischetti2018use}, which is able to solve the linearized version of our instances. Figure~\ref{fig:method_comparison} shows the ECDFs of the runtime and the final gaps at the end of the time limit. 
It can be seen that the \texttt{best} settings significantly improve their \texttt{base} counterparts. This is particularly pronounced for the B\&C algorithm, where the \texttt{base} setting solves only 55\% of instances to optimality, whereas its \texttt{best} counterpart increases this number to almost 90\%. Similar (but not so drastic) improvements are obtained for the cutting-plane method too. Finally, the overall best-performing approach is \texttt{BC-best} and the solver \MIX is also outperformed by both the cutting-plane algorithm and the B\&C.   

Tables~\ref{tab:method_comparison_cover} and~\ref{tab:method_comparison_MKP} provide additional insights into this comparison. For each of the five methods, we report the following average values: the runtime in seconds ($t$), the final gap (\texttt{Gap}), the final gap with respect to the best-known upper bound (\texttt{Gap*}), the root gap (\texttt{RGap}), the root gap with respect to the best-known upper bound (\texttt{RGap*}), the number of DCs separated at integer points (\texttt{nICut}), the number of DCs separated at fractional points (\texttt{nFCut}), the number of cuts where at least one redundant disjunctions was removed (\texttt{nRed}), the time needed to solve the follower problem ($t_F$), the additional time needed to separate a DC ($t_S$), the number of instances solved to optimality and the total number of instances considered in each row (\texttt{nSol}). 

%\red{Recall} that for each instance \texttt{Gap} is calculated as $100 (z^*-LB)/z^*$, where $z^*$ and $LB$ are the best-known objective-function value in the current setting and the lower bound, respectively. 
%In particular, 
The gaps are calculated as follows: 
\texttt{RGap} is calculated as $100 (z^*_R-LB_R)/z^*_R$, where $z^*_R$ and $LB_R$ are the best objective-function value and the lower bound at the end of the root node, respectively. In the $*$ counterparts of \texttt{Gap} and \texttt{RGap}, we use the best-known objective-function value of the instance over all the experiments described in this section, instead of $z^*$ and $z^*_R$. For the cutting-plane method, only a lower bound is available unless the instances is solved to optimality, thus we only provide \texttt{Gap*}.

%Finally, in column \texttt{nSol} we report the number of instances solved to optimality, followed by the total number of instances considered in each row.
In Table~\ref{tab:method_comparison_cover} each row presents average values over 10 instances with $n \in \{20,30,40,50\}$, and in Table~\ref{tab:method_comparison_MKP}, the instances are grouped according to the percentage of items that are controlled by the leader ($n_1/n \in \{0.5,0.75\}$). We observe that the removal of redundant disjunctions is particularly effective when fractional points are separated, and that for the cutting-plane method the redundant disjunctions are rarely detected. The latter can be explained by the fact that the problems solved to detect redundancy are more likely to be infeasible when considering local variable bounds. Moreover, we observe  speed-ups of orders of magnitude when using \texttt{G} instead of using \texttt{O}. In terms of final gaps, for the most difficult instances (namely those from QBCov with $n = 50$ and all instances from QBMKP), the best method is \texttt{BC-best}, providing final average gaps which are two to 15 times lower than the respective gaps of the competing methods.

%\paragraph{SOCP-based DCs with multiple disjunctions.}
\paragraph{Performance of the B\&C on instances with two linking constraints.}

\begin{figure}[tbp]
\caption{ECDFs reporting runtimes and final gaps for \texttt{BC-base} and \texttt{BC-best} over %140 (
123 binary instances with two linking constraints.
\label{fig:multidisj} }
\begin{center}

\includegraphics[width=0.5\textwidth]{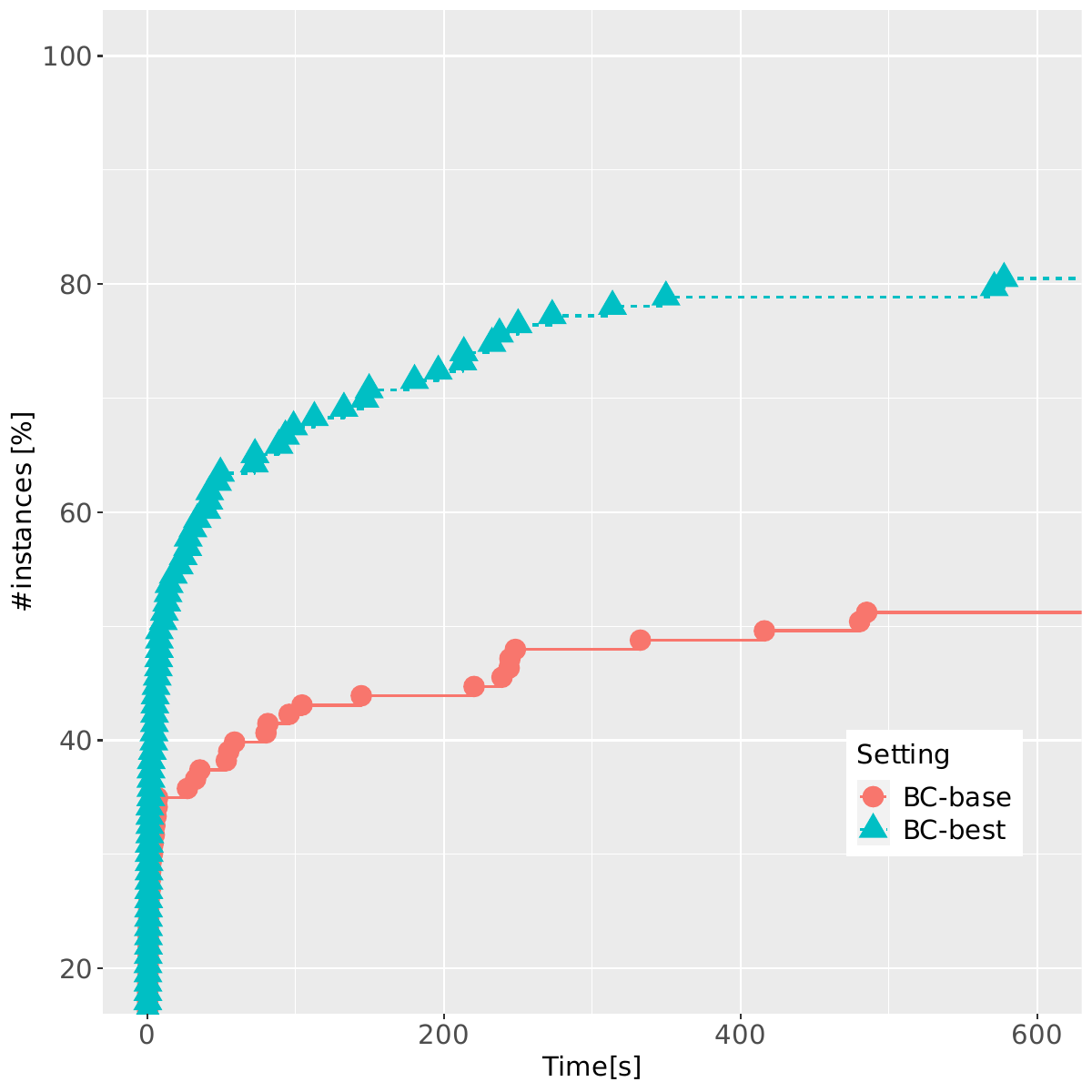}%
\includegraphics[width=0.5\textwidth]{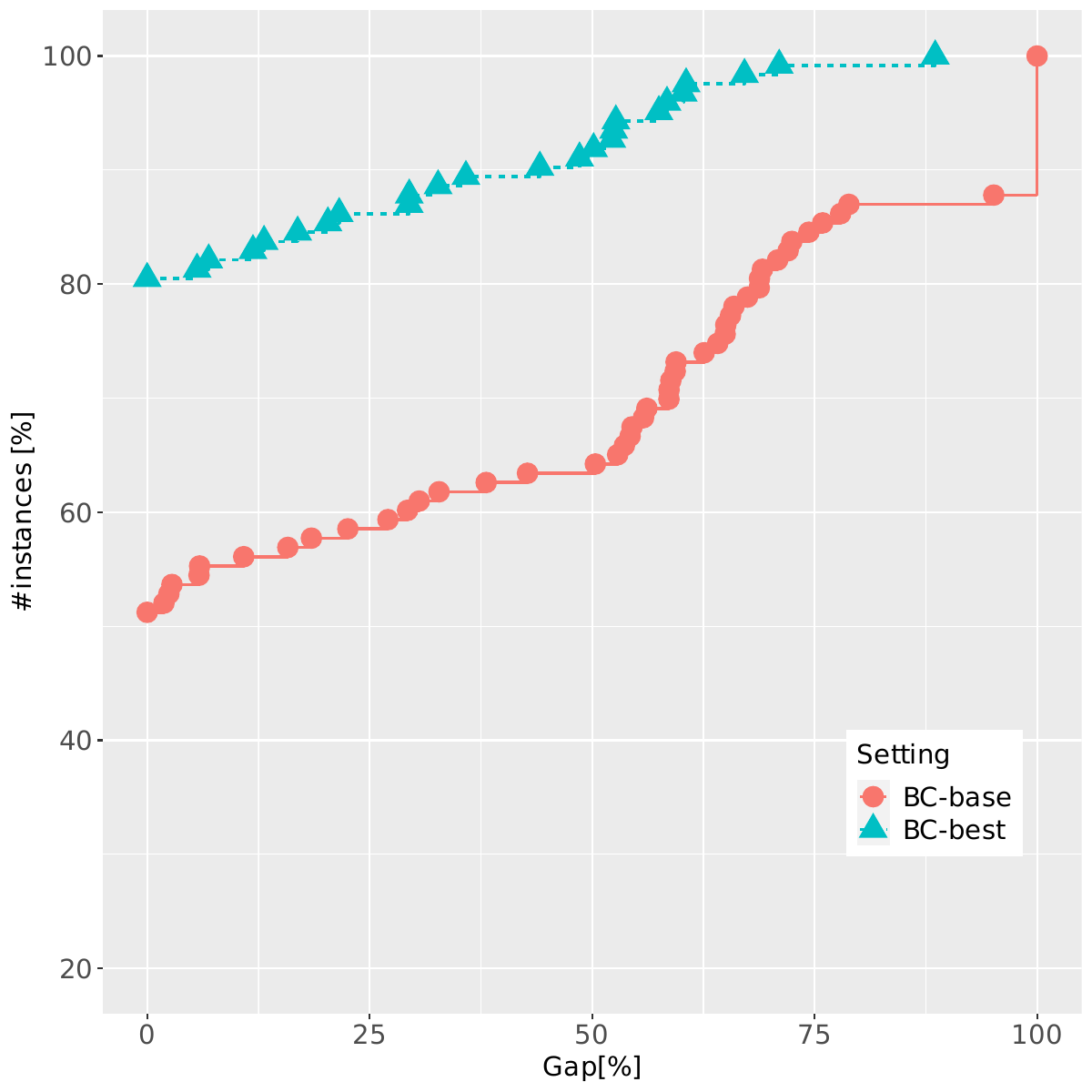}%

\end{center}
\end{figure}

We now turn our attention to the set of 140 binary instances with two linking constraints,  40 of them being from the benchmark set  QBCov and  100 from QBMKP. For these instances, 
%our DC separation procedure results in three disjunctions, 
\eqref{CG-SOCP} has three disjunctions,
one for the objective function, and one for each of the linking constraints. 
%\red{($m_2=2$)} \todo{should we just say two}. 
In order to have a fair comparison, out of 140 instances from this benchmark set, we only consider 123 instances for which at least one of the methods was either able to find a feasible solution, or to prove  infeasibility (for 17 instances from QBMKP, the feasibility status remains unknown). We again compare the settings \texttt{BC-best} and \texttt{BC-base} and report the corresponding ECDFs in Figure~\ref{fig:multidisj}. Also here, significant improvements in the performance can be achieved thanks to a proper combination of separation and disjunction removal strategies (around 50\% of instances are solved to optimality in the \texttt{BC-base} setting versus more than 80\% when \texttt{BC-best} is used instead). 
%We also notice that multiple disjunctions do negatively affect the performance of~\eqref{CG-SOCP}, but this effect remains moderate, as long as a relatively small number of disjunctions is considered. 
%
More detailed results for these instances are provided 
%in Appendix~\ref{sec:appenixA}, 
in Tables~\ref{tab:multidisj_covering} and~\ref{tab:multidisj_knapsack}. It is clear that under \texttt{BC-best} we are able to decrease the solution times for each group of instances and to detect many redundant disjunctions.

\begin{table}[tbp]
\centering \setlength{\tabcolsep}{3pt}
\caption{
Average results of \texttt{BC-base} and \texttt{BC-best} over 40 %(40) 
QBCov instances with two linking constraints.
\label{tab:multidisj_covering} 
}
\begingroup\scriptsize
\begin{tabular}{@{}rrrrrrrrrrrrrr@{}}
  \hline
  \multicolumn{1}{c}{$n$} & \multicolumn{1}{c}{\texttt{Setting}} & \multicolumn{1}{c}{$t$} & \multicolumn{1}{c}{\texttt{Gap}} & \multicolumn{1}{c}{\texttt{Gap*}} & \multicolumn{1}{c}{\texttt{RGap}}  & \multicolumn{1}{c}{\texttt{RGap*}} & \multicolumn{1}{c}{\texttt{nNode}} & \multicolumn{1}{c}{\texttt{nICut}} & \multicolumn{1}{c}{\texttt{nFCut}} & \multicolumn{1}{c}{\texttt{nRed}} & \multicolumn{1}{c}{$t_F$} & \multicolumn{1}{c}{$t_S$} & \multicolumn{1}{c}{\texttt{nSol}} \\ \hline
\multirow{2}[1]{*}{20} & \texttt{BC-base} & 2.0 & 0.0 & 0.0 & 90.0 & 26.3 & 291.1 & 108.9 & - & - & 1.2 & 0.7 & 10/10 \\ 
   & \texttt{BC-best} & 2.0 & 0.0 & 0.0 & 90.0 & 25.7 & 174.6 & 26.1 & 90.6 & 139.4 & 0.8 & 0.7 & 10/10 \\ 
  \multirow{2}[1]{*}{30} & \texttt{BC-base} & 225.0 & 14.8 & 10.5 & 90.0 & 27.4 & 3919.3 & 1643.0 & - & - & 211.0 & 13.3 & 7/10 \\ 
   & \texttt{BC-best} & 42.5 & 0.0 & 0.0 & 73.6 & 24.0 & 5280.2 & 216.1 & 3425.3 & 5379.6 & 19.8 & 20.7 & 10/10 \\ 
  \multirow{2}[1]{*}{40} & \texttt{BC-base} & 339.3 & 32.0 & 12.7 & 73.3 & 21.9 & 4469.1 & 1530.4 & - & - & 321.7 & 16.8 & 5/10 \\ 
   & \texttt{BC-best} & 102.0 & 0.0 & 0.0 & 80.0 & 21.3 & 8361.7 & 575.6 & 5255.2 & 6928.8 & 44.6 & 54.0 & 10/10 \\ 
  \multirow{2}[1]{*}{50} & \texttt{BC-base} & 550.4 & 56.0 & 36.9 & 100.0 & 40.3 & 3299.2 & 1218.1 & - & - & 532.7 & 17.0 & 1/10 \\ 
   & \texttt{BC-best} & 364.0 & 26.1 & 25.3 & 100.0 & 39.4 & 21184.3 & 1114.2 & 15539.7 & 17114.4 & 147.3 & 206.4 & 5/10 \\ 
   \hline
\end{tabular}
\endgroup
\end{table}

\begin{table}[tbp]
\centering \setlength{\tabcolsep}{3pt}
\caption{
Average results of \texttt{BC-base} and \texttt{BC-best} over 
%100.0 
83 QBMKP instances with two linking constraints.
\label{tab:multidisj_knapsack} 
}
\begingroup\scriptsize
\begin{tabular}{@{}rrrrrrrrrrrrrr@{}}
  \hline
  \multicolumn{1}{c}{$n_1/n$} & \multicolumn{1}{c}{\texttt{Setting}} & \multicolumn{1}{c}{$t$} & \multicolumn{1}{c}{\texttt{Gap}} & \multicolumn{1}{c}{\texttt{Gap*}} & \multicolumn{1}{c}{\texttt{RGap}}  & \multicolumn{1}{c}{\texttt{RGap*}} & \multicolumn{1}{c}{\texttt{nNode}} & \multicolumn{1}{c}{\texttt{nICut}} & \multicolumn{1}{c}{\texttt{nFCut}} & \multicolumn{1}{c}{\texttt{nRed}} & \multicolumn{1}{c}{$t_F$} & \multicolumn{1}{c}{$t_S$} & \multicolumn{1}{c}{\texttt{nSol}} \\ \hline
\multirow{2}[1]{*}{0.50} & \texttt{BC-base} & 417.8 & 40.1 & 19.8 & 86.0 & 26.0 & 3776.9 & 1387.8 & - & - & 404.8 & 12.3 & 14/33 \\ 
   & \texttt{BC-best} & 191.4 & 8.2 & 8.2 & 79.3 & 16.9 & 10128.7 & 524.2 & 9623.2 & 12297.6 & 82.1 & 102.8 & 26/33 \\ 
  \multirow{2}[1]{*}{0.75}  & \texttt{BC-base} & 295.6 & 27.8 & 14.4 & 67.9 & 17.7 & 15587.6 & 3987.9 & - & - & 246.7 & 46.2 & 26/50 \\ 
   & \texttt{BC-best} & 163.0 & 9.1 & 8.9 & 58.4 & 15.4 & 9742.1 & 385.6 & 8017.7 & 10365.4 & 71.2 & 86.4 & 38/50 \\ 
   \hline
\end{tabular}
\endgroup
\end{table}

\begin{figure}[tbp]
\caption{ECDFs reporting runtimes and final gaps for \texttt{BC-base} and \texttt{BC-best} over %100 (
97 integer instances with one linking constraint.
\label{fig:integer} }
\begin{center}

\includegraphics[width=0.5\textwidth]{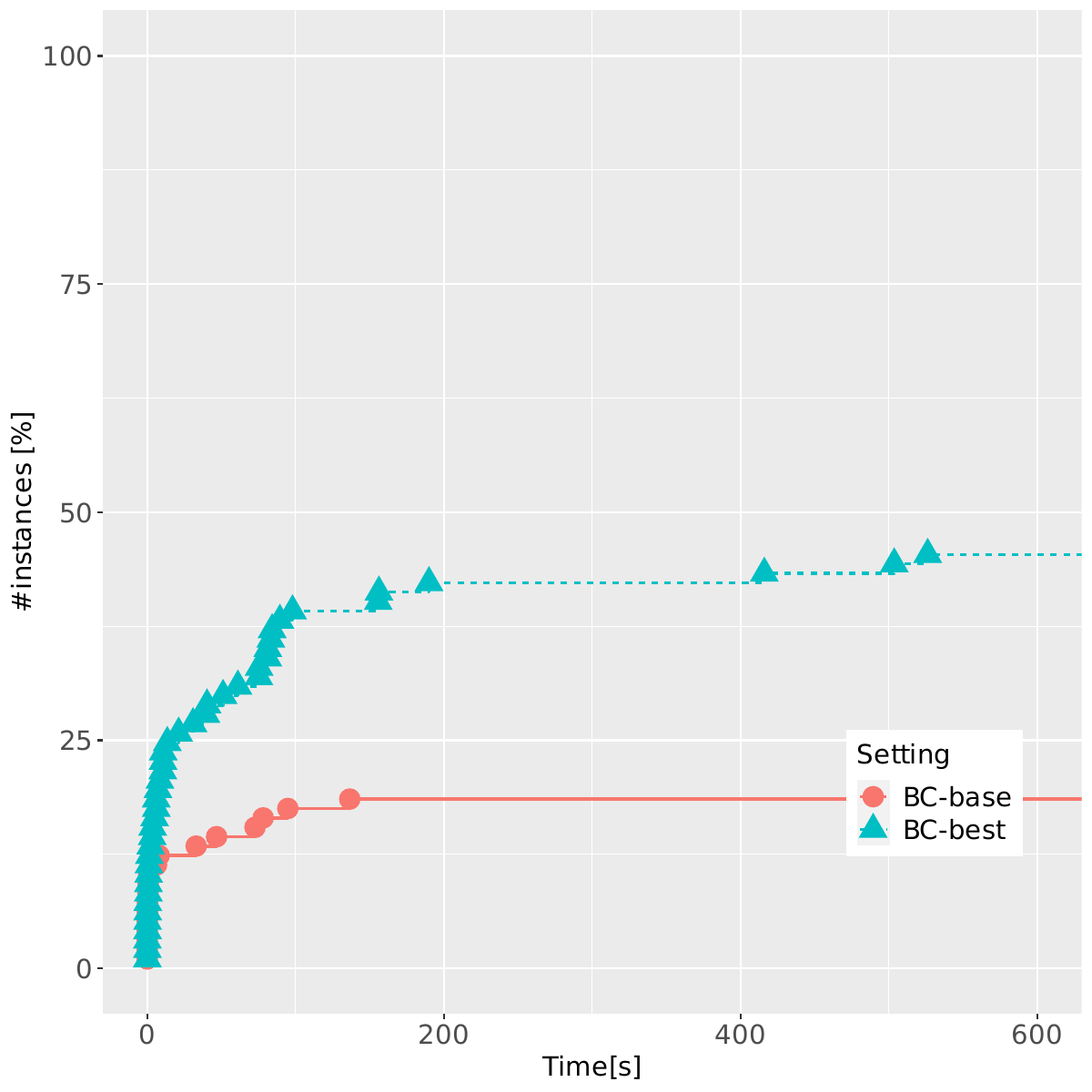}%
\includegraphics[width=0.5\textwidth]{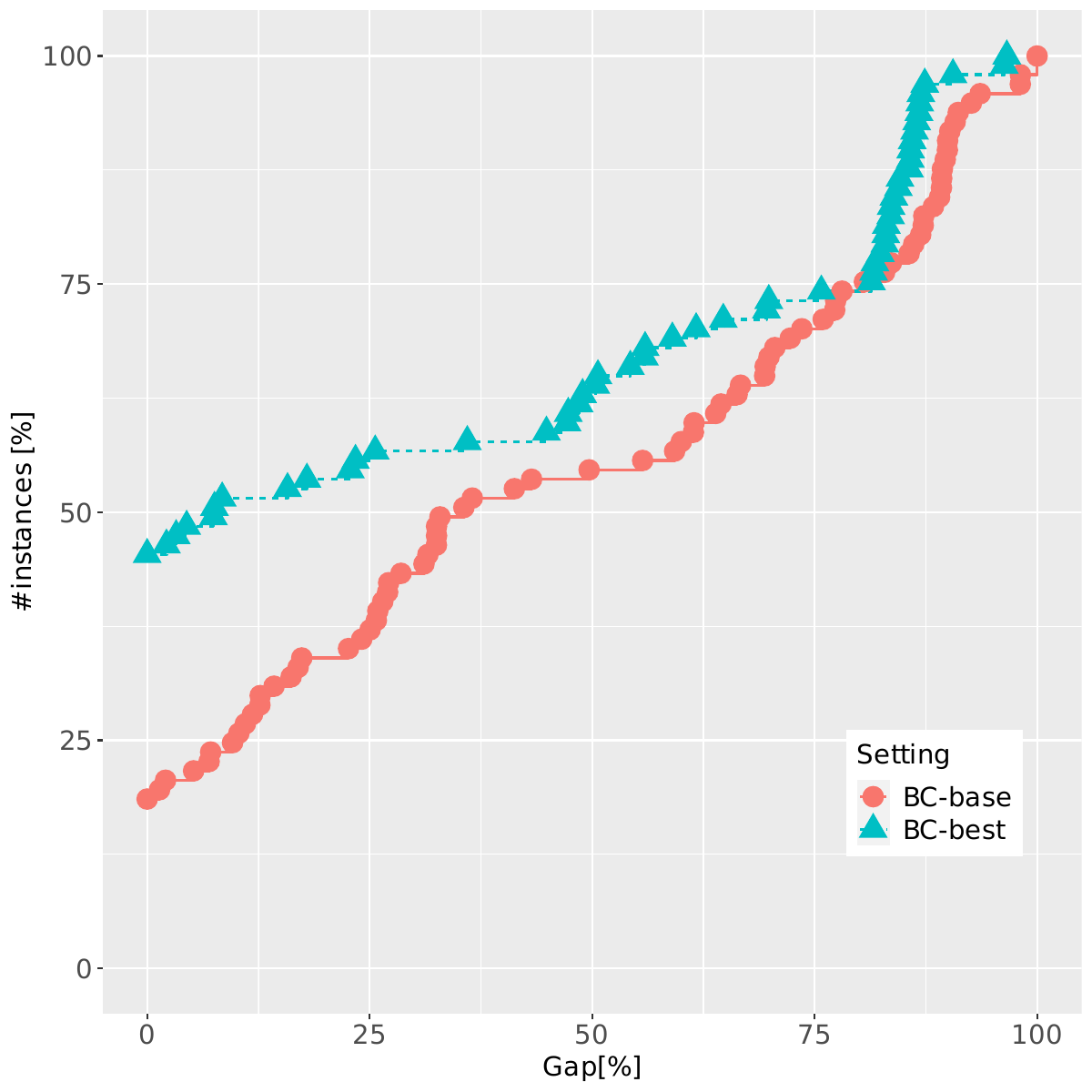}%
\end{center}
\end{figure}

\paragraph{Performance of the B\&C on instances with integer variables.}
Finally, we also consider 100 additional instances from the benchmark set QBMKP with integer variables. For 97 instances (for which we were able to either find a feasible solution, or to prove infeasibility)  the results comparing the settings \texttt{BC-best} and \texttt{BC-base} are summarized in Figure~\ref{fig:integer}. We observe that \IBNPs with integer variables are more difficult for our method than their binary counterparts. The improvements obtained using the \texttt{BC-best} setting are still significant (we double the number of instances solved to optimality within the time limit), however, more than 50\% of the instances from this benchmark set remain unsolved. These instances are the most difficult ones considered in this study, which can be explained by the much larger size of the search space. Based on the detailed results which are provided in Table~\ref{tab:integer}, we observe that
%For example, 
the number of explored branching nodes is orders of magnitude higher compared to the similar instances with binary variables, and the strength of DCs (in terms of the root bounds) is significantly weaker when integer variables are involved. 
%More detailed results are provided in 
%Appendix~\ref{sec:appenixA}, 

\begin{table}[tbp]
\centering  \setlength{\tabcolsep}{3pt}
\caption{Average results of \texttt{BC-base} and \texttt{BC-best} over 
%100 (97) 
97 integer QBMKP instances with one linking constraint.
\label{tab:integer} }
\begingroup\scriptsize
\begin{tabular}{@{}rrrrrrrrrrrrrr@{}}
  \hline
  \multicolumn{1}{c}{$n_1/n$} & \multicolumn{1}{c}{\texttt{Setting}} & \multicolumn{1}{c}{$t$} & \multicolumn{1}{c}{\texttt{Gap}} & \multicolumn{1}{c}{\texttt{Gap*}} & \multicolumn{1}{c}{\texttt{RGap}}  & \multicolumn{1}{c}{\texttt{RGap*}} & \multicolumn{1}{c}{\texttt{nNode}} & \multicolumn{1}{c}{\texttt{nICut}} & \multicolumn{1}{c}{\texttt{nFCut}} & \multicolumn{1}{c}{\texttt{nRed}} & \multicolumn{1}{c}{$t_F$} & \multicolumn{1}{c}{$t_S$} & \multicolumn{1}{c}{\texttt{nSol}} \\ \hline
\multirow{2}[1]{*}{0.50}  & \texttt{BC-base} & 517.6 & 64.6 & 53.3 & 96.0 & 55.7 & 43044.1 & 13409.9 & - & - & 411.0 & 98.6 & 7/47 \\ 
  & \texttt{BC-best} & 464.6 & 45.8 & 43.9 & 95.9 & 49.5 & 36864.1 & 959.6 & 35208.5 & 13368.0 & 180.0 & 255.5 & 11/47 \\ 
  \multirow{2}[1]{*}{0.75}  & \texttt{BC-base} & 471.1 & 25.9 & 21.6 & 86.4 & 23.8 & 123843.5 & 25398.6 & - & - & 246.5 & 204.2 & 11/50 \\ 
   & \texttt{BC-best} & 258.4 & 21.4 & 10.2 & 84.3 & 19.2 & 18200.1 & 1990.8 & 15386.6 & 5604.7 & 97.0 & 148.9 & 33/50 \\ 
   \hline
\end{tabular}
\endgroup
\end{table}

\section{Conclusions and outlook\label{sec:conclusions}}
In this article, we demonstrated that SOCP-based \DCs are an effective and promising methodology for solving a challenging family of discrete \BPs with a convex quadratic objective and linear constraints in the follower problem.   
Although \DCs have been employed with some success for several classes of MINLPs, their use and development for IBNPs is novel. 
The fact that we significantly outperform a state-of-the-art method for \MIBLOPs (after linearizing the nonlinear terms) indicates that further development of dedicated solution approaches for \IBNPs exploiting nonlinear (and in particular SOCP-based) techniques is a promising endeavour. 

There are still many open questions for future research. 
%From the computational  perspective, dealing with multiple linking constraints in the follower problem  requires an implementation of a SOCP-based separation procedure based on multi-disjunctions. 
%Our implementation could also be extended to deal with 
%second-order cone constraints at the leader. Moreover, 
The proposed \BC 
could be enhanced by bilevel-specific preprocessing, or 
bilevel-specific valid inequalities (as this has been done for \MIBLOPs in 
e.g., 
%\cite{fischetti2017new,fischetti2018use}
\cite{fischetti2017new,fischetti2018use}). 
Problem-specific strengthening 
inequalities could be used within disjunctions to obtain stronger DCs, and 
finally outer-approximation could be used as an alternative to SOCP-based 
separation. 
It also remains open to study problem generalizations involving (discrete) follower problems with (multiple) conic constraints. 

\begin{acknowledgements}
This research was funded in whole, or in part, by the Austrian Science Fund (FWF) [P 35160-N]. For the purpose of open access, the author has applied a CC BY public copyright licence to any Author Accepted Manuscript version arising from this submission.
It is also supported by the Johannes Kepler University Linz, Linz Institute of Technology (Project LIT-2019-7-YOU-211) and the JKU Business School. J. Lee was supported on this project by ESSEC and by ONR grant N00014-21-1-2135. 
\end{acknowledgements}
% ----------------------------------------------------------------
%\newpage
\bibliographystyle{splncs04}
\bibliography{disj}

%\clearpage
%\appendix

%\section{Computational Results, Continued} 
%\label{sec:appenixA}



% \begin{thebibliography}{}
% %
% % and use \bibitem to create references. Consult the Instructions
% % for authors for reference list style.
% % Format for Journal Reference
% Author, Article title, Journal, Volume, page numbers (year)
% Author, Book title, page numbers. Publisher, place (year)
% \end{thebibliography}

\end{document}